%% file: nls-composite.tex
\documentclass[final]{siamart220329}   % Use [final] when ready.
\usepackage{bbold}
\usepackage{bm}
\usepackage[caption=false]{subfig}

\usepackage{siampaper}
\usepackage{tikzscale}
\usetikzlibrary{backgrounds}
\usepackage{pgfplots}
\pgfplotsset{compat=1.16}
\usepgfplotslibrary{fillbetween}
\usepackage{adjustbox}

\usepackage{longtable}

\newcommand*{\includetikzgraphics}[2][]{%
  \includegraphics[#1]{#2}
}

\newcommand{\cahiernumber}{58}  % Insert your Cahier du GERAD number.
\title{
  A Levenberg-Marquardt Method for Nonsmooth Regularized Least Squares%
}

\author{
  Aleksandr Y. Aravkin%
  \thanks{%
    Department of Applied Mathematics,
    University of Washington,
    Seattle WA., USA\@.
    E-mail: \mailto{saravkin@uw.edu}.
  }
  \and
  Robert Baraldi%
  \thanks{%
    Optimization and Uncertainty Quantification,
    Sandia National Laboratories,
    P.O. Box 5800, Albuquerque, NM, 87125, USA\@.
    E-mail: \mailto{rjbaral@sandia.gov}.
    This research was sponsored by the Department of Energy
    Office of Science, Office of Advanced Scientific Computing Research's
    John von Neumann Fellowship.
    Sandia National Laboratories is a multimission laboratory
    managed and operated by National Technology and Engineering
    Solutions of Sandia, LLC., a wholly owned subsidiary of
    Honeywell International, Inc., for the U.S.\ Department of
    Energy’s National Nuclear Security Administration under
    contract DE-NA0003525.
    This paper describes objective technical results and analysis. Any
    subjective views or opinions that might be expressed in the paper
    do not necessarily represent the views of the U.S.\ Department of
    Energy or the United States Government.\@.
  }
  \and
  Dominique Orban%
  \thanks{%
    GERAD and Department of Mathematics and Industrial Engineering,
    Polytechnique Montr\'eal, QC, Canada.
    E-mail: \mailto{dominique.orban@gerad.ca}.
    Research partially supported by an NSERC Discovery Grant.
  }
}
\date{\today}

\input{defs}

\begin{document}

  % \linenumbers
  \maketitle

  \thispagestyle{firstpage}
  \pagestyle{myheadings}

  \begin{abstract}
    We develop a Levenberg-Marquardt method for minimizing the sum of a smooth nonlinear least-squares term \(f(x) = \tfrac{1}{2} \|F(x)\|_2^2\) and a nonsmooth term \(h\).
    Both \(f\) and \(h\) may be nonconvex.
    Steps are computed by minimizing the sum of a regularized linear least-squares model and a model of \(h\) using a first-order method such as the proximal gradient method.
    We establish global convergence to a first-order stationary point of both a trust-region and a regularization variant of the Levenberg-Marquardt method under the assumptions that \(F\) and its Jacobian are Lipschitz continuous and \(h\) is proper and lower semi-continuous.
    In the worst case, both methods perform \(O(\epsilon^{-2})\) iterations to bring a measure of stationarity below \(\epsilon \in (0, 1)\).
    We report numerical results on three examples: a group-lasso basis-pursuit denoise example, a nonlinear support vector machine, and parameter estimation in neuron firing.
    For those examples to be implementable, we describe in detail how to evaluate proximal operators for separable \(h\) and for the group lasso with trust-region constraint.
    In all cases, the Levenberg-Marquardt methods perform fewer outer iterations than a proximal-gradient method with adaptive step length and a quasi-Newton trust-region method, neither of which exploit the least-squares structure of the problem.
    Our results also highlight the need for more sophisticated subproblem solvers than simple first-order methods.
  \end{abstract}



  \input{introduction}
  \input{background}

  \input{lls}
  \input{nls}
  \input{prox}

  \input{numerical}

  \input{conclusion}

  \bibliographystyle{abbrvnat}
  \bibliography{abbrv,references}

  % \clearpage
  % \tableofcontents
  % \listoftodos\relax

\end{document}

%% file: defs.tex
\newtheorem{problemassumption}{Problem Assumption}[section]
\newtheorem{modelassumption}{Model Assumption}[section]
\newtheorem{stepassumption}{Step Assumption}[section]

\crefname{problemassumption}{Problem Assumption}{Problem Assumptions}
\Crefname{problemassumption}{Problem Assumption}{Problem Assumptions}
\crefname{modelassumption}{Model Assumption}{Model Assumptions}
\Crefname{modelassumption}{Model Assumption}{Model Assumptions}
\crefname{stepassumption}{Step Assumption}{Step Assumptions}
\Crefname{stepassumption}{Step Assumption}{Step Assumptions}

\Crefname{subsection}{Section}{Sections}
\crefname{subsection}{section}{sections}

\newcommand{\proj}[1]{\mathop{\textup{proj}}_{#1}}

\newcommand{\dom}{\mathop{\textup{dom}}}
\newcommand{\interior}{\mathop{\textup{int}}}
\newcommand{\B}{\mathbb{B}}

%% file: introduction.tex
%!TEX root = nls-composite.tex
\section{Introduction}

We consider the problem
\begin{equation}
  \label{eq:nlls}
  \minimize{x} \ f(x) + h(x),
  \qquad
  f(x) = \tfrac{1}{2} \|F(x)\|_2^2,
\end{equation}
where \(F: \R^n \to \R^m\) is continuously differentiable and \(h: \R^n \to \R\) is proper and lower semi-continuous; we allow $h$ to be nonsmooth and nonconvex.
In practice, $f$ is often a data-misfit term while $h$ is a regularizer designed to promote desirable properties in the solution, such as sparsity.
Numerous applications investigated in the nonsmooth regularized optimization literature actually have the structure~\eqref{eq:nlls}, including basis pursuit denoising \citep{donoho2006compressed,tibshirani1996regression}, sparse factorization and dictionary learning \citep{bach-jenatton-mairal-obozinski-2011}, and sparse total least squares \citep{zhu-leus-giannakis-2011}.
Yet nonsmooth numerical methods do not exploit the least-squares structure, nor accommodate general nonsmooth regularizers.
% \smarttodo{Fair to say? - couldn't find anything.}

We describe two methods for \eqref{eq:nlls}: a quadratic regularization variant and trust-region variant inspired by the method of \citet{levenberg-1944} and \citet{marquardt-1963}, denoted {\tt LM} and {\tt LMTR} respectively.
Steps are computed by approximately minimizing simpler nonsmooth iteration-dependent Gauss-Newton-type models.
Our algorithmic realizations utilize first-order methods, such as the proximal gradient method or the quadratic regularization method of \citet{aravkin-baraldi-orban-2021}, to solve the subproblems.
The trust-region approach allows for any arbitrary trust-region norm, which, in practice, is influenced by nonconvex subproblem tractibility.
% Similarly to the nonsmooth context of \citet{aravkin-baraldi-orban-2021}, we utilize one iteration of proximal gradient descent to serve the same role as the gradient in smooth optimization.
For both algorithms, we establish global convergence in terms of an optimality measure describing achievable decrease by a single proximal gradient step.
Additionally, we derive a worst-case complexity bound of $\mathcal{O}(1/\epsilon^2)$ iterations to bring the stationarity measure below a tolerance of $\epsilon\in(0,1)$ for {\tt LM} and {\tt LMTR}, i.e., the presence of a nonsmooth term in the objective yields a complexity bound of the same order as in the smooth case.

We provide implementation details and illustrate the performance of our methods on several numerical examples, including
basis pursuit denoise with group-lasso regularization,
nonlinear support vector machine with $\ell_{1/2}^{1/2}$-norm regularization,
and a sparse parameter estimation example taken from the Fitzhugh-Nagumo model of neuron firing.
Our methods exhibit favorable performance under certain conditions with respect to previous work \citet{aravkin-baraldi-orban-2021}.
We additionally provide efficient, open-source software implementations of {\tt LM} and {\tt LMTR} as a package in the  Julia language \cite{baraldi-orban-regularized-optimization-2022}.
We find that exploiting the least-squares structure yields few {\tt LM} and {\tt LMTR} outer iterations, a well-known benefit in smooth optimization.
The cost incurred is a large number of inner iterations, i.e, spent solving the subproblem.
Thus, the results highlight the need for more sophisticated methods to minimize the sum of a linear least-squares term and a nonsmooth regularizer.

\subsection*{Related research}

The present research is based on the framework laid out by \citet*{aravkin-baraldi-orban-2021}.
The convergence and complexity of our trust-region Levenberg-Marquardt implementation follow directly from the general results of \citep{aravkin-baraldi-orban-2021}.
To the best of our knowledge, the trust-region literature does not explicitly cover the case of a nonlinear least-squares smooth objective with a nonsmooth regularizer other than a penalty term even though numerous applications exhibit that structure.
See \citep{conn-gould-toint-2000} for background and an extensive treatment.

A large portion of the literature focuses on \(h\) convex and/or globally Lipschitz continuous, e.g., \citet{cartis-gould-toint-2011,grapiglia-yuan-yuan-2016} and references therein.
We do not attempt to give a comprehensive account of that literature here as we focus on significantly weaker assumptions.
While many methods exist in the first-order literature, e.g., \citep{combettes2011proximal}, few can effectively utilize any significant curvature information.
Proximal Newton methods \citep{lee-sun-saunders-2014} require solutions to nontrivial proximal operators and positive semi-definiteness of the Hessian.
The small number of references that allow both \(f\) and \(h\) to be nonconvex that we are aware of include:
\citet{li-lin-2015}, who design accelerations of the proximal gradient method under the assumption that \(f + h\) is coercive;
\citet{palm} who design an alterating method for cases where \(h(x) = h_1(x_1) + h_2(x_2)\) and \((x_1, x_2)\) is a partition of \(x\);
\citet{stella-themelis-sopasakis-patrinos-2017} who propose a linesearch limited-memory BFGS method named PANOC;
\citet{themelis-stella-patrinos-2017} who propose a nonmonotone linesearch proximal quasi-Newton method named ZeroFPR based on the forward-backward envelope;
and \citet{bot-csetnek-laszlo-2016}, who study a proximal method with momentum.
The last three converge if $f+h$ satisfies the Kurdyka-\L ojasiewicz (K\L) assumption.
Moreover, while all include~\eqref{eq:nlls} as a special case, few exploit any curvature information and none are specific to the least-squares structure.
The algorithms presented here, like those of \citep{aravkin-baraldi-orban-2021}, require no such coercivity or K\L\, assumptions.

\subsection*{Notation}

We use \(\|\cdot\|\) to represent a generic, but fixed, norm on \(\R^n\) or \(\R^m\).
The unit ball defined by that norm is \(\B\), and \(x + \Delta \B\) is the ball centered at \(x\) of radius \(\Delta > 0\).
For an integer \(q \geq 1\), \(\|\cdot\|_q\) is the \(\ell_q\)-norm and \(\B_q\) is the unit ball in the \(\ell_q\)-norm.
If \(A \subseteq \R^n\), \(\chi(\cdot \mid A)\) is the indicator of \(A\), i.e., the function whose value is \(0\) if \(x \in A\) and \(+\infty\) otherwise.
Unless otherwise noted, if \(A\) is a matrix, \(\|A\|\) denotes the spectral norm of \(A\), i.e., its largest singular value.
We use \(J(x):\R^n \to \R^{n\times m}\) to denote the Jacobian of \(F\) at \(x\).

%% file: background.tex
%!TEX root = nls-composite.tex
\section{Background}%
\label{sec:background}

\begin{definition}[Limiting subdifferential]
  Consider \(\phi: \R^n \rightarrow \overline{\R}\) and \(\bar{x} \in \R^n\) with \(\phi(\bar{x}) < \infty\).
  We say that \(v \in \R^n\) is a \emph{regular subgradient} of \(\phi\) at \(\bar{x}\), and we write \(v \in \hat \partial \phi(\bar{x})\) if
  \[
    % \phi(x) \geq \phi(\bar{x}) + \langle v, x - \bar{x} \rangle + o(\|x-\bar{x}\|)
    \liminf_{x \to \bar{x}} \, \frac{ \phi(x) - \phi(\bar{x}) - v^T (x - \bar{x}) }{ \|x - \bar{x}\|_2 } \geq 0.
  \]
  The set of regular subgradients is also called the \emph{Fr\'echet subdifferential}.
  We say that \(v\) is a \emph{general subgradient} of \(\phi\) at \(\bar{x}\), and we write \(v \in \partial \phi(\bar{x})\), if there are sequences \(\{x_k\}\) and \(\{v_k\}\) such that
  \[
    x_k \to \bar{x}, \quad \phi(x_k) \to \phi(\bar{x}), \quad v_k \in \hat \partial \phi(x^k) \text{ and } \ v^k \to v.
  \]
  The set of general subgradients is called the \emph{limiting subdifferential}.
\end{definition}

\begin{proposition}[{\protect \citealp[Theorem 10.1]{rtrw}}]%
  \label{prop:rtrw}
  If \(\phi: \R^n \to \overline{\R}\) is proper and has a local minimum at \(\bar{x}\), then \(0 \in \hat\partial \phi(\bar{x}) \subseteq \partial \phi(\bar{x})\).
  If \(\phi\) is convex, the latter condition is also sufficient for \(\bar{x}\) to be a global minimum.
  If \(\phi = f + h\) where \(f\) is continuously differentiable on a neighborhood of \(\bar{x}\) and \(h\) is finite at \(\bar{x}\), then  \(\partial \phi(\bar{x}) = \nabla f(\bar{x}) + \partial h(\bar{x})\).
\end{proposition}

If \(0 \in \hat{\partial}\phi(\bar{x})\), we say that \(\bar{x}\) is \emph{first-order stationary} for \(\phi\).
Under our assumptions,
\begin{equation}%
  \label{eq:nlls-stationary}
  x \text{ is first-order stationary for~\eqref{eq:nlls}}
  \quad \Longleftrightarrow \quad
  0 \in J{(x)}^T F(x) + \partial h(x).
\end{equation}

The proximal gradient method \citep{fukushima-mine-1981} applied to a regularized objective \(f(x) + h(x)\) where \(f\) is differentiable is defined by the iteration
\begin{equation}%
  \label{eq:prox-iteration}
  x_{k+1} \in \prox{\nu h}(x_k - \nu \nabla f(x_k)) \qquad (k \geq 0),
\end{equation}
where \(\nu > 0\) is a steplength and the \emph{proximal operator} is defined as
\begin{equation}%
  \label{eq:def-prox}
  \prox{\nu h}(y) := \argmin{u} \tfrac{1}{2} \|u - y\|_2^2 + \nu h(u).
\end{equation}
Without further assumptions on \(h\),~\eqref{eq:def-prox} is a set that may be empty, or contain one or more elements.
The iteration~\eqref{eq:prox-iteration} has the following descent property

\begin{lemma}[{\protect \citealp[Lemma~\(2\)]{palm}}]%
  \label{lem:prox-decrease}
  Let \(\nabla f\) be Lipschitz continuous with Lipschitz constant \(L \geq 0\), \(h\) be proper lower semi-continuous and \(\inf h > -\infty\).
  Let \(x_k \in \dom h\), \(0 < \nu < 1/L\), and \(x_{k+1}\) be defined according to~\eqref{eq:prox-iteration}.
  Then,
  \begin{equation}%
    \label{eq:prox-decrease}
    (f+h)(x_{k+1}) \leq
    (f+h)(x_k) - \tfrac{1}{2} (\nu^{-1} - L) \|x_{k+1} - x_k\|_2^2.
  \end{equation}
\end{lemma}

%% file: lls.tex
\section{Linear Least Squares}%
\label{sec:lls}

For fixed \(\sigma \geq 0\) and \(x \in \R^n\), define
\begin{subequations}
  \begin{align}
    \varphi(s; x) & := \tfrac{1}{2} \|J(x) s + F(x)\|_2^2,
    \label{eq:lls-gn} \\
    \psi(s; x) & \approx h(x + s) \quad  \text{with} \quad \psi(0; x) = h(x),
    \label{eq:lls-psi} \\
    m(s; x, \sigma) & := \varphi(s; x) + \tfrac{1}{2} \sigma \|s\|_2^2 + \psi(s; x).
    \label{eq:lls-lm-regularized}
  \end{align}
\end{subequations}
Consider the parametric problem and its optimal set
\begin{subequations}%
  \label{eq:lls-parametric}
  \begin{align}
    \label{eq:lls-parametric-val}
    p(x, \sigma) & := \min_{s} \ m(s; x, \sigma) \leq \varphi(0; x) + \psi(0; x) = f(x) + h(x) \\
    \label{eq:lls-parametric-opt}
    P(x, \sigma) & := \argmin{s} \ m(s; x, \sigma).
  \end{align}
\end{subequations}
The form of~\eqref{eq:lls-parametric} is representative of a Levenberg-Marquardt subproblem for~\eqref{eq:nlls} in which \(f\) and \(h\) are modeled separately.

In particular, \(\varphi(0; x) = f(x)\) and \(\nabla_s \varphi(0; x) = \nabla f(x)\).
We make the following additional assumption.
\begin{modelassumption}%
  \label{asm:parametric}
  For any \(x \in \R^n\), \(\psi(\cdot; x)\) is proper, lsc and prox-bounded, i.e., there exists \(\lambda_x \in \R_+ \cup \{+\infty\}\) such that \(\psi(\cdot; x) + \tfrac{1}{2} \lambda_x^{-1} \|\cdot\|_2^2\) is bounded below.
  In addition, \(\psi(0; x) = h(x)\), and \(\partial \psi(0; x) = \partial h(x)\).
\end{modelassumption}

In \Cref{asm:parametric}, we assume that our choice of \(\lambda_x\) is the supremum of all possible choices, and we refer to it as the \emph{threshold of prox-boundedness} of \(\psi(\cdot; x)\).
In particular, \(\psi(\cdot; x)\) is bounded below if and only if \(\lambda_x = +\infty\).

By \Cref{prop:rtrw}, if \(\sigma \geq \lambda_x^{-1}\),
\[
  s \in P(x, \sigma) \quad \Longrightarrow \quad 0 \in  \nabla \varphi(s; x) + \sigma s + \partial \psi(s; x).
\]
We define
\begin{equation}
  \label{eq:def-xi}
  \xi(x, \sigma) := (f+h)(x) - p(x, \sigma).
\end{equation}
The following stationarity criterion follows directly from the definitions above.

\begin{lemma}%
  \label{lem:lm-stationary}
  Let \Cref{asm:parametric} be satisfied and \(\sigma \geq \lambda_x^{-1}\).
  Then \( \xi(x, \sigma) = 0 \Longleftrightarrow 0 \in P(x, \sigma) \Longrightarrow x \) is first-order stationary for~\eqref{eq:nlls}.
  In addition, \(x\) is first-order stationary for~\eqref{eq:nlls} if and only if \(s = 0\) is first-order stationary for~\eqref{eq:lls-lm-regularized}.
\end{lemma}

\begin{proof}
  Note first that \( \xi(x, \sigma) = 0 \Longleftrightarrow p(x, \sigma) = (f+h)(x) = \varphi(0; x) + \psi(0; x) \), which occurs if and only if \(0 \in P(x, \sigma)\).
  \Cref{prop:rtrw} then implies \(0 \in \partial m(0; x, \sigma) = \nabla \varphi(0; x) + \partial \psi(0; x)\) and is equivalent to~\eqref{eq:nlls-stationary}.
\end{proof}

The next result states some properties of~\eqref{eq:lls-parametric}.

\begin{proposition}%
  \label{prop:lls-parametric}
  Let \Cref{asm:parametric} be satisfied.
  \(\dom p = \dom P = \dom \psi \times \{\sigma \mid \sigma \geq \lambda_x^{-1}\}\).
  In addition, for any \(x \in \R^n\),
  \begin{enumerate}
    \item\label{itm:p-lsc} \(p(x, \cdot)\) is proper lsc and for each \(\sigma > \lambda_x^{-1}\), \(P(x, \sigma)\) is nonempty and compact;
    \item\label{itm:P-osc} if \(\{\sigma_k\} \to \bar{\sigma} > \lambda_x^{-1}\) in such a way that \(\{p(x, \sigma_k)\} \to p(x, \bar{\sigma})\), and for each \(k\), \(s_k \in P(x, \sigma_k)\), then \(\{s_k\}\) is bounded and all its limit points are in \(P(x, \bar{\sigma})\);
    % \item\label{itm:P-single-valued} if \(\varphi(\cdot; x) + \psi(\cdot; x)\) is strictly convex, \(P(\Delta; x)\) is single-valued;
    \item\label{itm:osc-sufficient} \(p(x, \cdot)\) is continuous at any \(\bar{\sigma} > \lambda_x^{-1}\) and \(\{p(x, \sigma_k)\} \to p(x, \bar{\sigma})\) holds in part~\ref{itm:P-osc} if \(\bar{\sigma} > 0\).
  \end{enumerate}
\end{proposition}

\begin{proof}
  Parts~\ref{itm:p-lsc}--\ref{itm:P-osc} follow from applying \citep[Theorem~\(1.17\)]{rtrw} by noting that~\eqref{eq:lls-lm-regularized} is level-bounded in \(s\) locally uniformly in \((x, \sigma)\) because \(\psi(\cdot; x) + \tfrac{1}{2} \lambda_x^{-1} \|s\|_2^2\) is bounded and \(\varphi(s; x) + \tfrac{1}{2} (\sigma - \lambda_x^{-1}) \|s\|_2^2\) is level bounded in \(s\) locally uniformy in \((x, \sigma)\).
  Part~\ref{itm:osc-sufficient} also follows from \citep[Theorem~\(1.17\)]{rtrw} by noting that~\eqref{eq:lls-lm-regularized} is continuous in \(\sigma\) at any \(\bar{\sigma} > \lambda_x^{-1}\).
\end{proof}

By \Cref{prop:lls-parametric} part~\ref{itm:osc-sufficient}, \(\xi(x, \cdot)\) is continuous at any \(\bar{\sigma} > \lambda_x^{-1}\).

Although~\eqref{eq:lls-gn} is a natural model of \(f\) about \(x\), convergence properties may be stated in terms of the simpler first-order model
\begin{subequations}%
  \begin{align}
    \label{eq:lls-lin}
    \varphi_1(s; x) & := f(x) + \nabla f(x)^T s = \tfrac{1}{2} \|F(x)\|_2^2 + {(J(x)^T F(x))}^T s, \\
    \label{eq:lls-lin-regularized}
    m_1(s; x, \sigma) & := \varphi_1(s; x) + \tfrac{1}{2} \sigma \|s\|^2 + \psi(s; x).
  \end{align}
\end{subequations}

The first step of the proximal gradient method~\eqref{eq:prox-iteration} applied to the minimization of both \(\varphi(s; x) + \psi(s; x)\) and \(\varphi_1(s; x) + \psi(s; x)\) with steplength \(\nu > 0\) is
\begin{align}
  \label{eq:def-s1}
  s_1 & \in \prox{\nu \psi(\cdot; x)} (-\nu J{(x)}^T F(x))
  \\  & = \argmin{s} \ \tfrac{1}{2} \|s + \nu J{(x)}^T F(x)\|_2^2 + \nu \psi(s; x) \nonumber
  \\ & = \argmin{s} \ (J{(x)}^T F(x))^T s + \tfrac{1}{2} \nu^{-1} \|s\|_2^2 + \psi(s; x) \nonumber
  \\ & = \argmin{s} \ m_1(s; x, \nu^{-1}). \nonumber
\end{align}

If \(\nu^{-1} \geq \sigma\), then \(m_1(s; x, \sigma) \leq m_1(s; x, \nu^{-1})\).
Therefore, if \(s_1\) results from~\eqref{eq:def-s1}, it also induces decrease in~\eqref{eq:lls-lin-regularized}.

In parallel to \Cref{lem:lm-stationary} and \Cref{prop:lls-parametric}, we may define
\begin{subequations}
  \begin{align}
    \label{eq:lls-lin-parametric-val}
    p_1(x, \sigma) & := \min_s \ m_1(s; x, \sigma) \leq \varphi_1(0; s) + \psi(0; x) = f(x) + h(x) \\
    \label{eq:lls-lin-parametric-opt}
    P_1(x, \sigma) & := \argmin{s} \ m_1(s; x, \sigma), \\
    \label{eq:def-xi1}
    \xi_1(x, \sigma) & := (f + h)(x) - p_1(x, \sigma) \geq 0,
  \end{align}
\end{subequations}
and we have the following results, stating corresponding properties of \(p_1\) and \(\xi_1\).
The proofs replicate those in \Cref{prop:lls-parametric} and \Cref{lem:lm-stationary-1}. 

\begin{lemma}%
  \label{lem:lm-stationary-1}
  Let \Cref{asm:parametric} be satisfied and \(\sigma \geq \lambda_x^{-1}\).
  Then \( \xi_1(x, \sigma) = 0 \Longleftrightarrow 0 \in P_1(x, \sigma) \Longrightarrow x \) is first-order stationary for~\eqref{eq:nlls}.
  In addition, \(x\) is first-order stationary for~\eqref{eq:nlls} if and only if \(s = 0\) is first-order stationary for~\eqref{eq:lls-lin-regularized}.
\end{lemma}

\begin{proposition}%
  \label{prop:lls-parametric-1}
  Let \Cref{asm:parametric} be satisfied.
  \(\dom p_1 = \dom P_1 = \dom \psi \times \{\sigma \mid \sigma \geq \lambda_x^{-1}\}\).
  In addition, for any \(x \in \R^n\),
  \begin{enumerate}
    \item\label{itm:p1-lsc} \(p_1(x, \cdot)\) is proper lsc and for each \(\sigma > \lambda_x^{-1}\), \(P_1(x, \sigma)\) is nonempty and compact;
    \item\label{itm:P1-osc} if \(\{\sigma_k\} \to \bar{\sigma} > \lambda_x^{-1}\) in such a way that \(\{p_1(x, \sigma_k)\} \to p_1(x, \bar{\sigma})\), and for each \(k\), \(s_k \in P_1(x, \sigma_k)\), then \(\{s_k\}\) is bounded and all its limit points are in \(P_1(x, \bar{\sigma})\);
    % \item\label{itm:P-single-valued} if \(\varphi(\cdot; x) + \psi(\cdot; x)\) is strictly convex, \(P(\Delta; x)\) is single-valued;
    \item\label{itm:osc-sufficient-1} \(p_1(x, \cdot)\) is continuous at any \(\bar{\sigma} > \lambda_x^{-1}\) and \(\{p_1(x, \sigma_k)\} \to p_1(x, \bar{\sigma})\) holds in part~\ref{itm:P1-osc} if \(\bar{\sigma} > 0\).
  \end{enumerate}
\end{proposition}

Because \(L = 0\) for \(\varphi_1\), \Cref{lem:prox-decrease} implies that the decrease achieved by \(s_1\) is \((\varphi_1 + \psi)(s_1; x) \leq (\varphi_1 + \psi)(0; x) - \tfrac{1}{2} \nu^{-1} \|s_1\|^2\), which can be rearranged as
\begin{equation}%
  \label{eq:lin-model-decrease}
  (f + h)(x) - (\varphi_1 + \psi)(s_1; x) \geq
  \tfrac{1}{2} \nu^{-1} \|s_1\|^2 \geq
  \tfrac{1}{2} \sigma \|s_1\|^2.
\end{equation}

In the special case where \(\psi = 0\), \(s_1 = - \nu^{-1} \nabla f(x)\), so that~\eqref{eq:lin-model-decrease} reduces to
\[
  \xi_1(x, \sigma) \geq
  \xi_1(x, \nu^{-1}) \geq
  f(x) - \varphi_1(s_1; x) \geq
  \tfrac{1}{2} \sigma \nu^{-1} \|\nabla f(x)\|^2 \geq
  \tfrac{1}{2} \sigma^2 \|\nabla f(x)\|^2,
\]
which suggests that \(\sigma^{-1} (\xi_1(x, \nu^{-1}))^{1/2}\) may be used as stationarity measure.

%% file: nls.tex
%!TEX root = nls-composite.tex
\section{Nonlinear Least Squares}%
\label{sec:nls}

\subsection{A regularization approach}%
\label{sec:nls-lm}

We first examine the formulation of the method of \citeauthor{levenberg-1944} and \citeauthor{marquardt-1963} in which the model~\eqref{eq:lls-lm-regularized} is employed to compute a step.
Specifically, consider \Cref{alg:lm-regularized}.
The step \(s_k\) is computed by approximately minimizing~\eqref{eq:lls-lm-regularized} in stage~\ref{alg:lm-regularized:step-computation} but the quality of the step is measured without taking the regularization term \(\tfrac{1}{2} \sigma_k \|s_k\|^2\) into account in stage~\ref{alg:lm-regularized:step-rhok}.
The subproblem step \(s_k\) may be computed by continuing the iterations of the proximal gradient method initialized at \(s_{k,1}\). This gives rise to one possible implementation of \Cref{alg:lm-regularized}.

\begin{algorithm}[ht]
  \caption[caption]{%
    Nonsmooth regularized Levenberg-Marquardt method.%
    \label{alg:lm-regularized}
  }
  \begin{algorithmic}[1]
    \State Choose constants \(0 < \eta_1 \leq \eta_2 < 1\) and \(0 < \gamma_3 \leq 1 < \gamma_1 \leq \gamma_2\).
    \State Choose \(x_0 \in \R^n\) where \(h\) is finite, \(\sigma_0 > 0\), compute \(F(x_0)\) and \(h(x_0)\).
    \For{\(k = 0, 1, \dots\)}
      \State\label{alg:lm-regularized:step-nuk}%
      Choose a steplength \(\nu_k < 1 / (\|J(x_k)\|^2 + \sigma_k)\).
      \State Compute \(s_{k,1}\) as defined in~\eqref{eq:def-s1} and \(\xi_1(x_k, \nu_k^{-1})\) as defined in~\eqref{eq:def-xi1}.
      \State Define \(m(s; x_k, \sigma_k)\) as in~\eqref{eq:lls-lm-regularized}.
      \State\label{alg:lm-regularized:step-computation}%
      Compute an approximate solution \(s_k\) of~\eqref{eq:lls-parametric-opt}.
      \State\label{alg:lm-regularized:step-rhok}%
      Compute the ratio
      \[
      \rho_k :=
      \frac{
        f(x_k) + h(x_k) - (f(x_k + s_k) + h(x_k + s_k))
      }{
        \varphi(0; x_k) + \psi(0; x_k) - (\varphi(s_k; x_k) + \psi(s_k; x_k))
      }.
      \]
      \State\label{alg:lm-regularized:step-accept}%
      If \(\rho_k \geq \eta_1\), set \(x_{k+1} = x_k + s_k\). Otherwise, set \(x_{k+1} = x_k\).
      \State\label{alg:lm-regularized:step-update}%
      Update the regularization parameter according to
      \[
        \sigma_{k+1} \in
        \begin{cases}
          [\gamma_3 \sigma_k, \, \sigma_k] & \text{ if } \rho_k \geq \eta_2,
          \\ [\sigma_k, \, \gamma_1 \sigma_k] & \text{ if } \eta_1 \leq \rho_k < \eta_2,
          \\ [\gamma_1 \sigma_k, \, \gamma_2 \sigma_k] & \text{ if } \rho_k < \eta_1.
        \end{cases}
      \]
    \EndFor
  \end{algorithmic}
\end{algorithm}

It may occur that \(\sigma_k \leq \lambda_{x_k}^{-1}\).
In such a case, \(\psi(s_k; x_k) = -\infty\) so that the rules of extended arithmetic imply \(\rho_k = 0\), whether \(h(x_k + s_k) = +\infty\) or is finite.
Thus \(s_k\) will be rejected at stage~\ref{alg:lm-regularized:step-accept} and \(\sigma_{k+1}\) will be chosen larger than \(\sigma_k\) at stage~\ref{alg:lm-regularized:step-update}.
After a finite number of such increases, \(\sigma_k\) will exceed \(\lambda_{x_k}^{-1}\) and a step with finite \(\psi(s_k; x_k)\) will result.

Our main working assumption is the following.

% \begin{problemassumption}%
%   \label{asm:FJ-lipschitz}
%   The residual \(F\) is \(\mathcal{C}^1\) on the level set \(\Omega := \{x \in \R^n \mid (f+h)(x) \leq (f+h)(x_0)\}\) and \(h\) is proper and lower semi-continuous.
% \end{problemassumption}

% An auxiliary assumption that is slightly stronger but likely to be required for our assumptions on the model to be satisfied is the following.

\begin{problemassumption}%
  \label{asm:FJ-lipschitz}
  The residual \(F\) and its Jacobian \(J\) are bounded and Lipschitz continuous on \(\Omega := \{x \in \R^n \mid (f+h)(x) \leq (f+h)(x_0)\}\) and \(h\) is proper and lower semi-continuous.
\end{problemassumption}

While \Cref{asm:FJ-lipschitz} is a strong demand on all of \(\R^n\) and, in particular, rules out the case of linear least squares, it is a common assumption in the convergence analysis of the Levenberg-Marquardt method.
If \(\Omega\) is a compact set, then \(F\) is Lipschitz continuous on \(\Omega\) if it is \(\mathcal{C}^1\) on \(\Omega\), and \(J\) is Lipschitz continuous on \(\Omega\) if \(F\) is \(\mathcal{C}^2\) on \(\Omega\).

Under \Cref{asm:FJ-lipschitz}, \(\nabla f\) is Lipschitz continuous on \(\Omega\), i.e., there exists \(L > 0\) such that
\begin{equation}%
  \label{eq:df-lipschitz}
  |f(x + s) - (f(x) + \nabla f{(x)}^T s)| \leq \tfrac{1}{2} L \|s\|_2^2
  \quad \text{for all }
  x, \ x + s \in \Omega.
\end{equation}

We emphasize that in what follows, knowledge of \(L\), or an estimate thereof, is not required.
Our next assumption on the model is the following.

\begin{modelassumption}%
  \label{asm:model-adequacy}
  There exists a constant \(\kappa_{\textup{m}} > 0\) such that for all \(x\) and \(s \in \R^n\), \(|(f + h)(x + s) - (\varphi + \psi)(s; x)| \leq \kappa_{\textup{m}} \|s\|^2\).
\end{modelassumption}

\Cref{asm:model-adequacy} is essentially an assumption on the nonsmooth part \(\psi\) of the model.
Indeed,~\eqref{eq:lls-gn} and~\eqref{eq:df-lipschitz} combine to yield
\begin{align*}
  |f(x + s) - \varphi(s; x)| & \leq
  |f(x + s) - (f(x) + \nabla f{(x)}^T s)| + \tfrac{1}{2} \|J(x) s\|^2|
  \\ & \leq \tfrac{1}{2} (L + \|J(x)\|^2) \|s\|^2.
\end{align*}
where we used the definition of \(f(x)\), the identity \(\nabla f(x) = J{(x)}^T F(x)\), and~\eqref{eq:df-lipschitz}.
Thus if \(J\) is bounded on \(\Omega\), we obtain
\[
  |f(x + s) - \varphi(s; x)| \leq \tfrac{1}{2} (L + \sup_{x \in \Omega} \|J(x)\|^2) \|s\|^2.
\]
In particular, \Cref{asm:model-adequacy} is satisfied with \(\kappa_{\textup{m}} = \tfrac{1}{2} (L + \sup_{x \in \Omega} \|J(x)\|^2)\) if we select \(\psi(s; x) := h(x + s)\).

We make the following additional assumption and say that \(\{\psi(\cdot; x_k)\}\) is \emph{uniformly prox-bounded}.

\begin{modelassumption}%
  \label{asm:unif-prox-bounded}
  There exists \(\lambda > 0\) such that \(\lambda_{x_k} \geq \lambda\) for all \(k \in \N\).
\end{modelassumption}

\Cref{asm:unif-prox-bounded} is satisfied if \(h\) itself is prox-bounded and we select \(\psi(s; x_k) := h(x_k + s)\) at each iteration.

Our first result ensures that \(\sigma_k\) is bounded above in \Cref{alg:lm-regularized}.

\begin{theorem}%
  \label{thm:sigmasucc}
  Let \Cref{asm:FJ-lipschitz,asm:parametric,asm:model-adequacy,asm:unif-prox-bounded} be satisfied, and let
  \begin{equation}
    \label{eq:def-sigmasucc}
    \sigma_{\textup{succ}} :=
    \max(2 \kappa_{\textup{m}} / (1 - \eta_2), \, \lambda^{-1})
     > 0.
  \end{equation}
  If \(x_k\) is not first-order stationary and
  \(
    \sigma_k \geq \sigma_{\textup{succ}}
  \),
  then iteration \(k\) is very successful and
  \(
    \sigma_{k+1} \leq \sigma_k
  \).
\end{theorem}

\begin{proof}
  Let \(s_k\) be the step computed at iteration~\(k\) of \Cref{alg:lm-regularized}.
  If \(\sigma_k < \lambda_{x_k}^{-1}\), \(\rho_k = 0\) as explained above, \(s_k\) is rejected and \(\sigma_k\) is increased.
  Hence, we assume that \(\sigma_k \geq \lambda^{-1} \geq \lambda_{x_k}^{-1}\).
  Because \(x_k\) is not first-order stationary, \(s_k \neq 0\).
  Because \(s_k\) is an approximate solution of~\eqref{eq:lls-parametric-opt}, we must have
  \begin{equation*}
    \varphi(0; x_k) + \psi(0; x_k) \geq
    \varphi(s_k; x_k) + \tfrac{1}{2} \sigma_k \|s_k\|^2 + \psi(s_k; x_k)
  \end{equation*}
  and therefore,
  \begin{equation}%
    \label{eq:model-decrease}
    \varphi(0; x_k) + \psi(0; x_k) - (\varphi(s_k; x_k) + \psi(s_k; x_k)) \geq
    \tfrac{1}{2} \sigma_k \|s_k\|^2.
  \end{equation}

  \Cref{asm:model-adequacy} and~\eqref{eq:model-decrease} combine to yield
  \[
    |\rho_k - 1| =
    \frac{
      |f(x_k + s_k) + h(x_k + s_k) - (\varphi(s_k; x_k) + \psi(s_k; x_k))|
    }{
      \varphi(0; x_k) + \psi(0; x_k) - (\varphi(s_k; x_k) + \psi(s_k; x_k))
    } \leq
    \frac{2 \kappa_{\textup{m}} \|s_k\|^2}{\sigma_k \|s_k\|^2}.
  \]
  After simplifying by \(\|s_k\|^2\), we obtain \(\sigma_k \geq \sigma_{\text{succ}} \Longrightarrow \rho_k \geq \eta_2\).
\end{proof}

Note that \Cref{thm:sigmasucc} does not explicitly include \Cref{asm:FJ-lipschitz} in its assumptions, though it is likely to be required for \Cref{asm:model-adequacy} to hold.

Interestingly, \Cref{thm:sigmasucc} holds without assuming that the step \(s_k\) satisfies a sufficient decrease condition.
Upon examination of the proof, the reason turns out to be that any step that results in simple decrease in \(m(s; \sigma, x)\) results in sufficient decrease in \(\varphi(\cdot; x) + \psi(\cdot; x)\), independently of the method used to compute \(s_k\).

\Cref{thm:sigmasucc} ensures existence of a constant \(\sigma_{\max} > 0\) such that
\begin{equation}%
  \label{eq:sigmamax}
  \sigma_k \leq
  \sigma_{\max} := \min(\sigma_0, \gamma_2 \sigma_{\textup{succ}}) > 0
  \quad \text{for all } k \in \N.
\end{equation}

Our next result concerns the situation where a finite number of successful iterations occur.
The proof is almost identical to that of \citep[Theorem~\(6.4.4\)]{conn-gould-toint-2000} and \citep[Theorem~\(3.5\)]{aravkin-baraldi-orban-2021} and is omitted.

\begin{theorem}%
  \label{thm:lm-finite}
  Let \Cref{asm:FJ-lipschitz,asm:parametric,asm:model-adequacy} be satisfied.
  If \Cref{alg:lm-regularized} only generates finitely many successful iterations, then \(x_k = x^*\) for all sufficiently large \(k\) and \(x^*\) is first-order critical.
\end{theorem}

By \citet[Theorem~\(1.25\)]{rtrw}, \(p_1(x, \sigma)\) increases when \(\sigma\) increases, and thus, \(\xi_1(x, \sigma)\) decreases when \(\sigma\) increases.
Thus, it follows from~\eqref{eq:sigmamax} that
\begin{equation}%
  \label{eq:ximin}
  \xi_1(x_k, \sigma_k) \geq \xi_1(x_k, \sigma_{\max})
  \quad \text{for all } k \in \N.
\end{equation}
\Cref{lem:lm-stationary},~\eqref{eq:ximin} and the remarks at the end of \cref{sec:lls} suggest using \(\xi_1{(x_k, \sigma_{\max})}^{\frac12}\) as stationarity measure.
Indeed, for given \(\epsilon > 0\), \(\xi_1(x_k, \sigma_{\max}) \leq \epsilon / \sigma_{\max} \Longrightarrow \sigma_k \xi_1(x_k, \sigma_{\max}) \leq \epsilon\).

Because we must choose the steplength \(\nu_k\) as in Step~\ref{alg:lm-regularized:step-nuk} of \Cref{alg:lm-regularized}, we compute \(\xi_1(x_k, \nu_k^{-1})\) rather than \(\xi_1(x_k, \sigma_k)\).
Concretely, for given \(0 < \theta < 1\), we set
\begin{equation}%
  \label{eq:nuk}
  \nu_k := \theta / (\|J_k\|^2 + \sigma_k).
\end{equation}
Under \Cref{asm:FJ-lipschitz}, there exists \(\kappa_J > 0\) such that \(\|J(x)\| \leq \kappa_J\) for all \(x \in \Omega\).
Because \Cref{alg:lm-regularized} only generates \(x_k \in \Omega\), the above and~\eqref{eq:sigmamax} yield
\begin{equation}%
  \label{eq:numin}
  \nu_k \geq \theta / (\kappa_J^2 + \sigma_{\max}) := \nu_{\min} > 0
  \quad \text{for all } k \in \N.
\end{equation}
Therefore, \(\nu_k^{-1} \leq \nu_{\min}^{-1}\) for all \(k \geq 0\), and
\begin{equation}%
  \label{eq:ximin-nu}
  \xi_1(x_k, \nu_k^{-1}) \geq \xi_1(x_k, \nu_{\min}^{-1})
  \quad \text{for all } k \in \N.
\end{equation}

For a stopping tolerance \(\epsilon \in (0, \, 1)\), we seek to determine \(k(\epsilon) \in \N\) such that
\begin{equation}
  \label{eq:lm-stop}
  \xi_1{(x_k, \nu_{\min}^{-1})}^{\frac12} > \epsilon \quad \text{for all } k < k(\epsilon)
  \quad \text{and} \quad
  \xi_1{(x_{k(\epsilon)}, \nu_{\min}^{-1})}^{\frac12} \leq \epsilon.
\end{equation}
Define the sets
\begin{subequations}%
  \label{eq:S-U-sets}
  \begin{align}
       \mathcal{S} & := \{ k \in \N \mid \rho_k \geq \eta_1 \},
    \\ \mathcal{S}(\epsilon) & := \{ k \in \mathcal{S} \mid k < k(\epsilon) \},
    \\ \mathcal{U}(\epsilon) & := \{k \in \N \mid k \not \in \mathcal{S} \text{ and } k < k(\epsilon) \}.
  \end{align}
\end{subequations}

In order to conduct the complexity analysis, it is necessary to assume that the step computation at stage~\ref{alg:lm-regularized:step-computation} of \Cref{alg:lm-regularized} is related to \(\xi_1(x_k, \sigma_k)\).
We make the following assumption.

\begin{stepassumption}%
  \label{asm:step-computation}
  There exists \(\kappa_{\textup{mdc}} \in (0, \, 1)\) such that \(s_k\) computed at stage~\ref{alg:lm-regularized:step-computation} of \Cref{alg:lm-regularized} satisfies
  \begin{equation}%
    \label{eq:sufficient-decrease}
    \varphi(0; x_k) + \psi(0; x_k) - (\varphi(s_k; x_k) + \psi(s_k; x_k)) \geq
    \kappa_{\textup{mdc}} \xi_1(x_k, \nu_k^{-1}).
  \end{equation}
\end{stepassumption}

\Cref{asm:step-computation} is similar to sufficient decrease conditions used in trust-region methods---see \citep{conn-gould-toint-2000}.
\citet{aravkin-baraldi-orban-2021} provide a concrete use of such condition in a trust-region method for nonsmooth regularized optimization.
Clearly, the sufficient decrease assumption is satisfied after a single step of the proximal gradient method applied to~\eqref{eq:lls-lm-regularized}.
Hence, it is also satisfied at a minimizer of~\eqref{eq:lls-lm-regularized}.
Thus, in step~\ref{alg:lm-regularized:step-computation} of \Cref{alg:lm-regularized}, one strategy is to continue the proximal-gradient iterations until a stopping condition is attained.

The following results parallel those of \citet{aravkin-baraldi-orban-2021}, which are in turn inspired from those of \citet{cartis-gould-toint-2011} and references therein.

\begin{lemma}%
  \label{lem:lm-cmplx-successful}
  Let \Cref{asm:FJ-lipschitz,asm:parametric,asm:model-adequacy} be satisfied and \(s_k\) be computed according to \Cref{asm:step-computation}, where \(\nu_k\) is chosen according to~\eqref{eq:nuk}.
  Assume there are infinitely many successful iterations and that \(f(x) + h(x) \geq {(f + h)}_{\textup{low}}\) for all \(x \in \R^n\).
  Then, for all \(\epsilon \in (0, \, 1)\),
  \begin{equation}
    \label{eq:lm-bound-Seps}
    |\mathcal{S}(\epsilon)| \leq
    \frac{(f + h)(x_0) - (f + h)_{\textup{low}}}{\eta_1 \kappa_{\textup{mdc}} \epsilon^2} =
    O(\epsilon^{-2}).
  \end{equation}
\end{lemma}

\begin{proof}
  For \(k \in \mathcal{S}(\epsilon)\), \Cref{asm:step-computation} and~\eqref{eq:ximin-nu} imply
  \begin{align*}
    (f+h)(x_k) - (f+h)(x_k + s_k) & \geq \eta_1 (\varphi(0; x_k) + \psi(0; x_k) - (\varphi(s_k; x_k) + \psi(s_k; x_k)))
    \\ & \geq \eta_1 \kappa_{\textup{mdc}} \xi_1(x_k, \nu_k^{-1})
    \\ & \geq \eta_1 \kappa_{\textup{mdc}} \xi_1(x_k, \nu_{\min}^{-1})
    \\ & \geq \eta_1 \kappa_{\textup{mdc}} \epsilon^2.
  \end{align*}
  The rest of the proof mirrors that of \citep[Lemma~\(3.6\)]{aravkin-baraldi-orban-2021}.
\end{proof}

\begin{lemma}%
  \label{lem:lm-cmplx-unsuccessful}
  Under the assumptions of \Cref{lem:lm-cmplx-successful},
  \begin{equation}
    \label{eq:lm-bound-Ueps}
    |\mathcal{U}(\epsilon)| \leq
    \frac{\log(\sigma_{\max} / \sigma_0)}{\log(\gamma_1)} + |\mathcal{S}(\epsilon)| \frac{|\log(\gamma_3)|}{\log(\gamma_1)}
    =
    O(\epsilon^{-2}).
  \end{equation}
\end{lemma}

\begin{proof}
  For each \(k \in \mathcal{U}(\epsilon)\), \(\sigma_{k+1} \geq \gamma_1 \sigma_k\), while for each \(k \in \mathcal{S}(\epsilon)\), \(\sigma_{k+1} \geq \gamma_3 \sigma_k\).
  Thus if \(k(\epsilon)\) is the iteration for which~\eqref{eq:lm-stop} occurs for the first time,
  \[
    \sigma_0 \gamma_1^{|\mathcal{U}(\epsilon)|} \gamma_3^{|\mathcal{S}(\epsilon)|} \leq
    \sigma_{k(\epsilon) - 1} \leq
    \sigma_{\max}.
  \]
  Taking logarithms, we have
  \[
    |\mathcal{U}(\epsilon)| \log(\gamma_1) + |\mathcal{S}(\epsilon)| \log(\gamma_3) \leq
    \log(\sigma_{\max} / \sigma_0).
  \]
  Rearranging and recalling that \(0 < \gamma_3 < 1\) yields~\eqref{eq:lm-bound-Ueps}.
\end{proof}

Combining \Cref{lem:lm-cmplx-successful,lem:lm-cmplx-unsuccessful} yields the overall iteration complexity bound.

\begin{theorem}%
  \label{thm:lm-complexity-bound}
  Under the assumptions of \Cref{lem:lm-cmplx-successful},
  \begin{equation}
    \label{eq:lm-nonsmooth-complexity}
    |\mathcal{S}(\epsilon)| + |\mathcal{U}(\epsilon)| =
    O(\epsilon^{-2}).
  \end{equation}
\end{theorem}

Stated differently, \Cref{thm:lm-complexity-bound} ensures that either \((f + h)(x_k) \to -\infty\) or that \(\liminf_{k \to \infty} \xi_1(x_k, \nu_{\min}^{-1}) = 0\).

\subsection{A trust-region approach}%
\label{sec:nls-tr}

We now apply Algorithm~\(3.1\) of \citet{aravkin-baraldi-orban-2021} to~\eqref{eq:nlls}.
We assume that each \(f_i: \R^n \to \R\) is \(\mathcal{C}^1\), so that their Problem Assumption~\(3.1\) is satisfied.
A natural model for \(f\) about \(x\) is the Gauss-Newton model~\eqref{eq:lls-gn}, which satisfies \(\varphi(0; x) = f(x)\) and \(\nabla_s \varphi(0; x) = \nabla f(x) = J{(x)}^T F(x)\).
The model \(\psi(s; x)\) of \(h(x + s)\) is required to satisfy the same \Cref{asm:model-adequacy}, which holds provided \(\nabla f\) is Lipschitz continuous or each \(f_i\) is \(\mathcal{C}^2\) with bounded Hessian.
In \citet[Algorithm~\(3.1\)]{aravkin-baraldi-orban-2021}, the first proximal gradient step \(s_1\) is computed by solving
\begin{equation}%
  \label{eq:s1-tr}
  \begin{aligned}
    \minimize{s} \ & \tfrac{1}{2} \|F(x)\|_2^2 + {(J{(x)}^T F(x))}^T s + \tfrac{1}{2} \nu^{-1} \|s\|^2 + \psi(s; x)
    \\ \st & \ \|s\| \leq \Delta,
  \end{aligned}
\end{equation}
i.e.,
\[
  s_1 \in \prox{\nu \psi(\cdot; x) + \chi(\cdot \mid \Delta \B)} (-\nu J{(x)}^T F(x)),
\]
where \(0 < \nu < 1 / (\|J(x)\|^2 + \alpha^{-1} \Delta^{-1})\) for a preset constant \(\alpha > 0\).
Subsequent steps continue the proximal gradient iterations to compute an approximate solution of
\begin{equation}%
\label{eq:tr-subproblem}
  \minimize{s} \ \tfrac{1}{2} \|J(x) s + F(x)\|_2^2 + \psi(s; x)
  \quad \ \st \|s\| \leq \min(\beta \|s_1\|, \, \Delta),
\end{equation}
where \(\beta \geq 1\).
The above describes a trust-region variant of the method of \citet{levenberg-1944} and \citet{marquardt-1963} for regularized nonlinear least-squares problems.
The assumption that \(\psi(\cdot; x)\) is prox-bounded can be removed because \(\psi(\cdot; x) + \chi(\cdot \mid \Delta \B)\) is always bounded below, hence prox-bounded with \(\lambda_x = \infty\).
An approximate solution of~\eqref{eq:tr-subproblem} must satisfy \Cref{asm:step-computation} with \(\xi_1(x, \sigma)\) replaced with
\[
  \hat{\xi}_1(\Delta; x, \nu) := f(x) + h(x) - \hat{p}_1(\Delta; x, \nu),
\]
where \(\hat{p}_1(\Delta; x, \nu)\) is the optimal value of~\eqref{eq:s1-tr}.

Under the above assumptions, \citeauthor{aravkin-baraldi-orban-2021} establish that the trust-region radius \(\Delta\) never drops below the threshold
\[
  \Delta_{\min} :=
  \min\left(
    \Delta_0, \,
    \hat{\gamma}_1
    \frac{\kappa_{\textup{mdc}} (1 - \eta_2)}{2 \kappa_{\textup{m}} \alpha \beta^2}
  \right),
\]
where \(\Delta_0 > 0\) is the initial trust-region radius, \(\hat{\gamma}_1 \in (0, \, 1)\) is the fraction by which \(\Delta\) is reduced on rejected steps, \(\eta_2 \in (0, \, 1)\) is the threshold above which \(\Delta\) is increased on accepted steps, and \(\kappa_{\textup{mdc}}\) and \(\kappa_{\textup{m}}\) play similar roles as the constants of the same name in \Cref{asm:model-adequacy,asm:step-computation}.

\citeauthor{aravkin-baraldi-orban-2021} use \(\hat{\xi}_1(\Delta_{\min}; x, \nu)\) as stationarity measure.
They show that for any \(\epsilon \in (0, \, 1)\), the number of iterations necessary to achieve
\[
  \hat{\xi}_1(\Delta_{\min}; x, \nu)^{\frac12} \leq \epsilon
\]
is \(O(\epsilon^{-2})\) provided that \(f + h\) is bounded below.
We refer the reader to \citep{aravkin-baraldi-orban-2021} for complete details.

%% file: prox.tex
%!TEX root = nls-composite.tex
\section{Proximal operators}%
\label{sec:prox}

In \Cref{alg:lm-regularized} or the algorithm of \Cref{sec:nls-tr}, a typical model of the nonsmooth term \(h\) is \(\psi(s; x) := h(x + s)\).
If those algorithms are to use \citeauthor{aravkin-baraldi-orban-2021}'s quadratic regularization method \citep[Algorithm~\(6.1\)]{aravkin-baraldi-orban-2021} to compute a step, the latter will in turn form a model of \(\psi(\cdot; x)\) at each iteration.
In order to simplify notation, let \(\psi_k(s) := \psi(s; x_k) = h(x_k + s)\) be the model used at iteration \(k\) of \Cref{alg:lm-regularized} or the algorithm of \Cref{sec:nls-tr}.

\subsection{General proximal operators}

In \Cref{alg:lm-regularized}, the nonsmooth term in the objective of the subproblem is \(\psi_k(s)\).
The typical model about \(s_j\) reduces to \(\omega_j(t) = \psi_k(s_j + t) = h(x_k + s_j + t)\) and, instead of~\eqref{eq:qr-tr-subproblem}, the step computed is
\begin{equation}%
  \label{eq:qr-reg-subproblem}
  t_j \in \argmin{t} \ \tfrac{1}{2} \nu^{-1} \|t - q\|^2 + h(x_k + s_j + t).
\end{equation}
The same change of variable as above yields
\[
  v_j \in \argmin{v} \ \tfrac{1}{2} \nu^{-1} \|v - \bar{q}\|^2 + h(v) = \prox{\nu h}(\bar{q}),
\]
whether \(h\) is separable or not.
Thus we obtain
\[
  t_j \in \prox{\nu h}(\bar{q}) - (x_k + s_j).
\]

The nonsmooth term in the objective of the subproblem of the algorithm of \Cref{sec:nls-tr} is \(\psi_k(s) + \chi(s ; \Delta_k)\).
About iterate \(s_j\) of \citep[Algorithm~\(6.1\)]{aravkin-baraldi-orban-2021}, the user supplies a model \(\omega_j(t) := \omega(t; s_j) \approx \psi_k(s_j + t) + \chi(s_j + t \mid \Delta_k \B)\), and the typical choice is \(\omega_j(t) = \psi_k(s_j + t) + \chi(s_j + t \mid \Delta_k \B) = h(x_k + s_j + t) + \chi(s_j + t \mid \Delta_k \B)\).
The step computed is
\(
  t_j \in \prox{\nu \omega_j}(q)
\)
for certain fixed \(\nu > 0\) and \(q \in \R^n\), i.e.,
\begin{equation}%
  \label{eq:qr-tr-subproblem}
  t_j \in \argmin{t} \ \tfrac{1}{2} \nu^{-1} \|t - q\|^2 + h(x_k + s_j + t) + \chi(s_j + t \mid \Delta_k \B).
\end{equation}
The change of variables \(v := x_k + s_j + t\) allows us to rewrite~\eqref{eq:qr-tr-subproblem} as
\begin{equation}
  \label{eq:prototype}
    v_j \in \argmin{v} \ \tfrac{1}{2} \nu^{-1} \|v - \bar{q}\|^2 + h(v) + \chi(v - x_k \mid \Delta_k \B),
\end{equation}
where \(\bar{q} := x_k + s_j + q\), from which we recover \(t_j = v_j - (x_k + s_j)\).

\subsection{Separable shifted proximal operators}

If \(h\) is separable and the trust region is defined by the \(\ell_{\infty}\)-norm, the problem decomposes and the \(i\)-th component of \(v_j\) is
\begin{equation}
  \label{eq:qr-tr-subproblem-varsub}
\begin{aligned}
  v_{j,i} & \in \argmin{v_i} \ \tfrac{1}{2} \nu^{-1} (v_i - \bar{q}_i)^2 + h_i(v_i) + \chi(v_i - x_{k,i} \mid [-\Delta_k, \Delta_k])
  \\      & =   \argmin{v_i} \ \tfrac{1}{2} \nu^{-1} (v_i - \bar{q}_i)^2 + h_i(v_i) + \chi(v_i \mid [x_{k,i} - \Delta_k, x_{k,i} + \Delta_k]).
\end{aligned}
\end{equation}
Two situations may occur.
In the first situation, \(x_{k,i} - \Delta_k < v_{j,i} < x_{k,i} + \Delta_k\), so that
\(
  v_{j,i} \in \prox{\nu h_i}(\bar{q}_i),
\)
i.e.,
\[
  t_{j,i} \in \prox{\nu h_i}(\bar{q}_i) - (x_{k,i} + s_{j,i}).
\]
In the second situation, at least one unconstrained solution lies outside of \([x_{k,i} - \Delta_k, x_{k,i} + \Delta_k]\), so that constrained global minima of~\eqref{eq:qr-tr-subproblem-varsub} are either one or both bounds, and/or unconstrained local minima that lie between the bounds.

When \(h\) is convex, the constrained solution is the feasible point nearest the unique unconstrained global solution, i.e.,
\[
  v_{j,i} \in \proj{[x_{k,i} - \Delta_k, x_{k,i} + \Delta_k]}(\prox{\nu h_i}(\bar{q}_i)),
\]
i.e.,
\[
  t_{j,i} \in \proj{[x_{k,i} - \Delta_k, x_{k,i} + \Delta_k]}(\prox{\nu h_i}(\bar{q}_i)) - (x_{k,i} + s_{j,i}).
\]

\begin{example}[$\ell_{1/2}^{1/2}$ pseudonorm]
  Consider $\psi(s) = \|s \|_{1/2}^{1/2} = \sum_j |s_j|^{1/2}$.
  When the trust-region bounds are inactive, \citet{cao2013l12} express the solution of~\eqref{eq:qr-tr-subproblem-varsub}
  as
  \begin{align*}
    v_{j,i} = \begin{cases}
      \phantom{-}\tfrac23 |\bar q_i|\left( 1 + \cos \left( \tfrac23 \pi - \tfrac23 \mu_\lambda (\bar q_i)\right) \right) &\quad \phantom{|}\bar q_i\phantom{|} > p (\lambda)\\
      \phantom{-}0 & \quad |\bar q_i| \leq p(\lambda)\\
      -\tfrac23 |\bar q_i|\left( 1 + \cos \left( \tfrac23 \pi - \tfrac23 \mu_\lambda (\bar q_i)\right) \right) &\quad \phantom{|}\bar q_i\phantom{|} < - p (\lambda)
    \end{cases}
  \end{align*}
  where
  \begin{align*}
    \mu_\lambda(\bar q_i) \coloneqq \arccos\left( \frac{\lambda}{4} \left( \frac{|\bar q_i|}{3}\right)^{-3/2} \right), \qquad p(\lambda) \coloneqq \frac{54^{1/3}}{4}(2\lambda)^{2/3}.
  \end{align*}
  When the trust-region constraint is active,
  \citet{cao2013l12} state that the above yields the inflection points of~\eqref{eq:qr-tr-subproblem-varsub}.
  We simply check the inflection points as well as the bounds.
  If the inflection points are within the bounds, we choose the minimum; if not, we select the minimum value of the cost function at the bounds.
\end{example}

\subsection{Nonseparable shifted proximal operators for convex \boldmath{\(h\)}}

In this section we consider examples of nonseparable shifted proximal operators.
The starting point is~\eqref{eq:prototype} where we assume that $h$ is closed, proper, and convex.
We rewrite
\[
  \chi(v - x \mid \Delta \B) = \sup_z \, \langle v - x, z \rangle - \sigma_{\Delta \B} (z),
\]
where we write \(x\) and \(\Delta\) instead of \(x_k\) and \(\Delta_k\) for simplicity, and where the support function
\[
  \sigma_{\Delta \B} (z) := \sup_d \, \langle d, z \rangle + \chi(d\mid \Delta \B).
\]
We substitute into~\eqref{eq:prototype} and obtain the saddle point problem
\begin{equation}%
  \label{eq:saddle}
  \min_v \sup_z \, \tfrac{1}{2} \nu^{-1} \|v - \bar{q}\|^2 + h(v) + \langle v - x, z \rangle - \sigma_{\Delta \B} (z).
\end{equation}
The objective of~\eqref{eq:saddle} is convex in \(v\) and concave in \(z\).
The saddle-point conditions can be written
\begin{align*}
  0 & \in \nu^{-1} (v - \bar{q}) + \partial h(v) + z = \nu^{-1} (v - (\bar{q} - \nu z)) + \partial h(v) \\
  0 & \in v - x - \partial \sigma_{\Delta \B}(z).
\end{align*}
The first condition implies that $v \in \prox{\nu h}(\bar{q} - \nu z)$.
By convexity of \(h\), \(v\) is unique so that we are left with
\begin{equation}
\label{eq:genproxcond}
  0 \in v - x - \partial \sigma_{\Delta \B}(z),
  \quad \text{where} \quad
  \prox{\nu h}(\bar{q} - \nu z) = \{v\}.
\end{equation}

\subsubsection{Special case: \boldmath{\(\ell_2\)}-norm}

For $h(\cdot) := \lambda\|\cdot\|_2$,
\begin{equation}%
  \label{eq:2-prox}
  \prox{\nu \lambda \|\cdot\|_2}(y) =
  \begin{cases}
    0 & \text{ if } \|y\| \leq \nu \lambda \\
    \left(1 - \frac{\nu\lambda }{\|y\|_2}\right) y & \text{ if } \|y\| > \nu \lambda
  \end{cases}.
\end{equation}

We now show how to solve~\eqref{eq:prototype} by converting~\eqref{eq:genproxcond} to a scalar root finding problem.
For given \(z\), let
\[
  \zeta = \zeta(z) := \|\bar{q} - \nu z\|_2.
\]
There are two possibilities.

\noindent
\textbf{Case A}:
If $\zeta \leq \nu \lambda$,~\eqref{eq:2-prox} yields 
\[
  \prox{\nu\lambda \|\cdot\|_2}(\bar{q} - \nu z) = \{v\} = \{0\}.
\]
The optimal value of~\eqref{eq:prototype} in this case is $\tfrac{1}{2} \nu^{-1} \|\bar{q}\|^2$.

\noindent
\textbf{Case B}:
If $\zeta > \nu \lambda$,~\eqref{eq:2-prox} yields 
\begin{equation}%
  \label{eq:prox-case2}
  \prox{\nu \lambda \|\cdot\|_2}(\bar{q} - \nu z) =
  \{v\} =
  \left\{ \left(1- \frac{\nu\lambda}{\zeta }\right)(\bar{q} - \nu z) \right\},
\end{equation}
and~\eqref{eq:genproxcond} becomes
\begin{align*}
  0 & \in x - \left(1- \frac{\nu\lambda}{\zeta}\right)(\bar{q} - \nu z) + \partial \sigma_{\Delta \B}(z) \\
    & = (\zeta - \nu\lambda)\frac{\nu}{\zeta}\left(z - \left(\frac{1}{\nu}\bar{q} - \frac{\zeta}{\nu(\zeta -\nu\lambda)}x\right) \right)  + \partial \sigma_{\Delta \B}(z),
\end{align*}
which we interpret as
\begin{equation}%
  \label{eq:z-of-zeta}
  z = z(\zeta) :=
  \prox{\frac{\zeta}{\nu(\zeta-\nu\lambda)}\sigma_{\Delta \B} } \left(\frac{1}{\nu}\bar{q} - \frac{\zeta}{\nu(\zeta -\nu\lambda)}x\right).
\end{equation}
Recall that \citep[Theorem~\(6.46\)]{beckFO}
\begin{equation}%
  \label{eq:prox-support}
  \prox{\alpha \sigma_{\Delta \B}}(y) = y - \alpha \proj{\Delta \B}(\alpha^{-1} y),
  \quad (\alpha > 0).
\end{equation}
Therefore, the projection into \(\Delta \B\) must be computable.
In our implementation, we use \(\B = \B_\infty\).

We may now search for \(\zeta\) such that
\begin{equation}%
  \label{eq:root-finding}
  g(\zeta) := \zeta - \|\bar{q} - \nu z(\zeta)\|_2 = 0.
\end{equation}
Because projections into convex sets are Lipschitz continuous, so is \(g\) over \((\nu \lambda, +\infty)\).

Since~\eqref{eq:prototype} is strongly convex, there is a unique solution, and so \(g\) has at most one root such that $\zeta > \nu \lambda$.
Any such root of \(g\) yields \(v\) given by~\eqref{eq:prox-case2} and \(z(\zeta)\) given by~\eqref{eq:z-of-zeta} that jointly satisfy~\eqref{eq:genproxcond}.
If \(g\) has no such root, the Case~A must occur.

The combination of~\eqref{eq:z-of-zeta} and~\eqref{eq:prox-support} yields
\begin{equation}%
  \label{eq:q-nuz}
  \bar{q} - \nu z(\zeta) =
  \frac{\zeta}{\zeta - \nu \lambda} \left[ x + \proj{\Delta \B} \left( \frac{\zeta - \nu \lambda}{\zeta} \bar{q} - x \right) \right].
\end{equation}

As \(\zeta \uparrow \infty\), \((\zeta - \nu \lambda) / \zeta \uparrow 1\), and by continuity, the term between square brackets in~\eqref{eq:q-nuz} converges to \(x + \proj{\Delta \B}(\bar{q} - x)\).
Therefore, \(\|\bar{q} - \nu z(\zeta)\|_2 \to \|x + \proj{\Delta \B}(\bar{q} - x)\|_2\) and for sufficiently large \(\zeta\), we must have \(g(\zeta) > 0\).

To study \(g(\zeta)\) as \(\zeta \downarrow \nu \lambda\), we consider several mutually-exclusive cases.

\begin{enumerate}
  \item\label{itm:x-not-in-B} If \(x \not \in \Delta \B\), then, \(\proj{\Delta \B}(-x) \neq -x\).
  As \(\zeta \downarrow \nu \lambda\), \((\zeta - \nu \lambda) / \zeta \downarrow 0\), and by continuity, the term between square brackets converges to \(x + \proj{\Delta \B}(-x) \neq 0\).
  Therefore, \(\|\bar{q} - \nu z(\zeta)\|_2 \to \infty\) and for sufficiently small \(\zeta\), we must have \(g(\zeta) < 0\).

  \item\label{itm:x-interior} Consider next the case where \(x \in \interior \Delta \B\).
  For \(\zeta\) sufficiently close to \(\nu \lambda\),
  \begin{equation}%
    \label{eq:proj-inside}
    \proj{\Delta \B} \left(\frac{\zeta - \nu \lambda}{\zeta} \bar{q} - x \right) = \frac{\zeta - \nu \lambda}{\zeta} \bar{q} - x,
  \end{equation}
  and \(\bar{q} - \nu z(\zeta) = \bar{q}\), i.e., \(z(\zeta) = 0\).
  In this case,
  \begin{enumerate}
    \item\label{itm:q>nl} if \(\|\bar{q}\|_2 > \nu \lambda\), then \(g(\zeta) < 0\) for \(\zeta\) close enough to \(\nu \lambda\),
    \item\label{itm:q<=nl} if \(\|\bar{q}\|_2 \leq \nu \lambda\), then \(g(\zeta) > 0\) for all \(\zeta > \nu \lambda\);
  \end{enumerate}

  \item\label{itm:x=delta} If \(\|x\|_\infty = \Delta\) and \(\proj{\Delta \B}(\bar{q} - x) = -x\), then \(\proj{\Delta \B}(\alpha \bar{q} - x) = -x\) for any \(\alpha > 0\).
  In this case, the term between square brackets in~\eqref{eq:q-nuz} is always zero, and \(\bar{q} - \nu z(\zeta) = 0\).
  Thus for all \(\zeta > \nu \lambda\), \(g(\zeta) = \zeta > 0\).

  \item\label{itm:x=delta-noproj} If \(\|x\|_\infty = \Delta\) but \(\proj{\Delta \B}(\bar{q} - x) \neq -x\), there are two possible situations.
  Either the ray \(\alpha \bar{q} - x\) intersects \(\interior \Delta \B\), or it does not.
    If it does,~\eqref{eq:proj-inside} occurs for all \(\zeta\) sufficiently close to \(\nu \lambda\), \(\bar{q} - \nu z(\zeta) = \bar{q}\), and cases~\ref{itm:q>nl}--\ref{itm:q<=nl} apply.
  If it does not, we have from Lipschitz continuity that
  \[
    \hspace{-2em}
    \left\| x + \proj{\Delta \B} \left(\frac{\zeta - \nu \lambda}{\zeta} \bar{q} - x \right) \right\|_2 \! =
    \left\| \proj{\Delta \B} \left(\frac{\zeta - \nu \lambda}{\zeta} \bar{q} - x \right) - \proj{\Delta \B}(-x) \right\|_2 \! \leq
    \frac{\zeta - \nu \lambda}{\zeta} \|\bar{q}\|_2.
  \]
  Thus, \(\|\bar{q} - \nu z(\zeta)\|_2 \leq \|\bar{q}\|_2\), and
  \begin{enumerate}
    \item\label{itm:q>nl2} if \(\|\bar{q}\|_2 > \nu \lambda\), then \(g(\zeta) \geq \zeta - \|\bar{q}\|_2 > 0\) for \(\zeta > \|\bar{q}\|_2\), and so there may exist a root in \((\nu \lambda, \|\bar{q}\|_2]\).
      By~\eqref{eq:q-nuz}, and the fact that \(\|y\|_2 \leq \sqrt{n} \|y\|_\infty\) for all \(y\), we also have
      \[
        \hspace{-2.5em}
        \|\bar{q}\ - \nu z(\zeta)\|_2 \leq
        \frac{\zeta}{\zeta - \nu \lambda} \left( \|x\|_2 + \left\| \proj{\Delta \B} \left(\frac{\zeta - \nu \lambda}{\zeta} \bar{q} - x \right)  \right\|_2 \right) \leq
        \frac{(\|x\|_2 +  \Delta \sqrt{n}) \zeta}{\zeta - \nu \lambda},
      \]
      so that \(g(\zeta) > 0\) for \(\zeta > \nu \lambda + 2 \Delta \sqrt{n}\).
      Thus, the search interval may potentially be reduced to \((\nu \lambda, \min(\nu \lambda + \|x\|_2 + \Delta \sqrt{n}, \|\bar{q}\|_2)]\).
    \item\label{itm:q<=nl2} if \(\|\bar{q}\| \leq \nu \lambda\), then \(g(\zeta) > 0\) for all \(\zeta > \nu \lambda\).
  \end{enumerate}
\end{enumerate}

Thus, in cases~\ref{itm:x-not-in-B} and~\ref{itm:q>nl}, a root is guaranteed to exist in \((\nu \lambda, +\infty)\) and can be found by a bisection method.
The upper bound may be found by observing that~\eqref{eq:q-nuz} implies
\[
  \|\bar{q} - \nu z(\zeta)\| \leq \frac{\zeta}{\zeta - \nu} (\|x\| + \Delta),
\]
so that
\[
  g(\zeta) = \zeta - \|\bar{q} - \nu z(\zeta)\| \geq
  \zeta - \frac{\zeta}{\zeta - \nu \lambda} (\|x\| + \Delta),
\]
and \(g(\zeta) > 0\) as soon as \(\zeta > \|x\| + \Delta + \nu \lambda\).

In case~\ref{itm:x-not-in-B}, a lower bound follows by applying the reverse triangle inequality to~\eqref{eq:q-nuz}:
\[
  \|\bar{q} - \nu z(\zeta)\| \geq \frac{\zeta}{\zeta - \nu \lambda} (\|x\| - \Delta),
\]
so that \(g(\zeta) < 0\) as soon as \(\zeta < \nu \lambda + \|x\| - \Delta\).

In case~\ref{itm:q>nl}, the lower bound is simply \(\|\bar{q}\|\).

In cases~\ref{itm:q<=nl},~\ref{itm:x=delta} and~\ref{itm:q<=nl2}, there can be no root in \((\nu \lambda, +\infty)\) and Case~A must occur.

Only case~\ref{itm:q>nl2} requires a root search, with or without sign change.
If no root exists in the search interval, Case~A must occur.

\subsubsection{Special case: Group lasso}

The group lasso penalty is a sum of \(\ell_2\)-norms of subvectors:
\[
  R_g(x) = \sum_i \|x_{[i]}\|_2,
\]
where the $x_{[i]}$ partition $x$ into non-overlapping groups.
The proximal operator of $R_g$ consists in applying~\eqref{eq:2-prox} to each subvector:
\begin{equation}
  \label{eq:g-prox}
  \prox{\lambda R_g}(z)_{[i]}  =  \left(1 - \frac{\lambda}{\|z_{[i]}\|_2}\right)_+ z_{[i]}.
\end{equation}
Thus, the strategy of the previous section may be applied to each group.

%% file: numerical.tex
%!TEX root = nls-composite.tex
\section{Implementation and numerical experiments}%
\label{sec:numerical}

Our implementation of Algorithm~\(3.1\) of \citep{aravkin-baraldi-orban-2021} and \Cref{alg:lm-regularized} for~\eqref{eq:nlls} employs \citeauthor*{aravkin-baraldi-orban-2021}'s quadratic regularization method, named {\tt R2}, to compute a step.
{\tt R2} may be viewed as an implementation of the proximal gradient method with adaptive step size.
The trust-region variant uses \(\Delta_0 = 1\), terminates the outer iterations as soon as \(\xi(\Delta_k; x_k, \nu_k)^{1/2} < \epsilon_a + \epsilon_r \, \xi_{1,0}^{1/2}\), where \(\epsilon_a > 0\) and \(\epsilon_r > 0\) are an absolute and a relative tolerance, and \(\xi_{1,0}\) is the value of \(\xi_1\) observed at the first iteration.
A round of inner iterations terminates as soon as
\begin{equation}
  \label{eq:inner_stop}
  \hat{\xi}_1(x_k + s, \hat{\sigma}_k) \leq
  \begin{cases}
    10^{-1} & \text{if } k = 0, \\
    \max(\epsilon, \min(10^{-1}, \xi_1(x_k, \sigma_k) / 10)) & \text{if } k > 0,
  \end{cases}
\end{equation}
where \(\hat{\sigma}_k\) and \(\hat{\xi}_1\) are the regularization parameter and first-order stationarity measure used inside {\tt R2}.
In \Cref{alg:lm-regularized}, we use \(\sigma_0 = 0.01\), and we terminate the outer iterations as soon as \(\xi_1(x_k, \sigma_k)^{1/2} < \epsilon\) for a tolerance \(\epsilon > 0\) because \(\sigma_{\max}\) is unknown.
The inner iterations stop in the same manner as \eqref{eq:inner_stop}.
All algorithms are implemented in the Julia language \citep{bezanson-edelman-karpinski-shah-2017} version~\(1.8\) as part of the RegularizedOptimization.jl package \citep{baraldi-orban-regularized-optimization-2022}.
The shifted proximal operators are implemented in the ShiftedProximalOperators.jl package \citep{baraldi-orban-shifted-proximal-operators-2022}, while test problems are in the RegularizedProblems.jl package \citep{baraldi-orban-regularized-problems-2022}.
By contrast with the numerical results of \citet{aravkin-baraldi-orban-2021}, test cases are explicitly implemented as nonlinear least-squares problems, with access to the residual \(F(x)\) and its Jacobian, and not simply the gradient of \(f(x) := \tfrac{1}{2} \|F(x)\|_2^2\).
Jacobian-vector and transposed-Jacobian-vector products are either implemented manually or computed via forward \citep{forwarddiff} and reverse \citep{reversediff} automatic differentiation, respectively.

We perform comparisons with {\tt R2} and with the quasi-Newton trust-region method of \citet{aravkin-baraldi-orban-2021}, named {\tt TR}, and which does not exploit the structure of~\eqref{eq:nlls}.
The trust region is defined in \(\ell_\infty\)-norm and the quadratic model uses a limited-memory SR1 Hessian approximation with memory \(5\).
In all experiments, we use \(\psi(s; x) := h(x + s)\).

A direct comparison between the four methods is difficult because {\tt LM} and {\tt LMTR} do not utilize the same gradient; they instead take Jacobian-vector and transposed-Jacobian-vector products.
To provide a meaningful comparison, in the tables below, we state:
1) the number of objective (or residual) evaluations;
2) the number of gradient evaluations (for {\tt R2} and {\tt TR}) ;
3) the number of transposed-Jacobian-vector products (for {\tt LM} and {\tt LMTR}), listed under gradient evaluations;
4) the solve time in seconds.
% Hence, the number of Jacobian vector products ($Jv$) and transpose-Jacobian vector products ($J^Tv$) are explicitly subproblem-dependent.
Our rationale is as follows.
{\tt LM} and {\tt LMTR} pass a model to {\tt R2} whose objective evaluation requires one $Jv$, and whose gradient uses a $Jv$ and a $J^Tv$.
Note however that the latter $Jv$ can be cached and reused.
Thus, {\tt R2} requires one $Jv$ at each iteration, and additionally one $J^Tv$ at each successful iteration.
% It is therefore difficult to compare against algorithms directly, unless we know how gradients and $Jv$/$J^Tv$'s will be computed.

In the figures, we plot descent as a function of residual/objective evaluations.
% In forward mode automatic differentiation, \(\text{time}(Jv) \leq \frac{5}{2} \text{time}(F)\) and \(\text{mem}(Jv) \leq 2 \text{mem}(F)\).
% In reverse mode automatic differentiation, \(\text{time}(J^T u) \leq 4 \text{time}(F)\) and \(\text{mem}(J^T u) \leq 2 \text{mem}(F)\).
% This context reveals that residual evaluations are more expensive than Jacobian-vector products.

The summary of the numerical results below is that exploiting the least-squares structure results in a large reduction in outer iterations.
However, solving the subproblem with a first-order method such as {\tt R2} consumes many $J^T v$.
Our experiments thus highlight the need for more sophisticated subproblem solvers dedicated to~\eqref{eq:lls-lm-regularized} and~\eqref{eq:tr-subproblem}.

\subsection{Group LASSO}

In the group-LASSO problem, we observe noisy data from a linear system \(b = A x_{T} + \varepsilon\), where \(A \in \R^{m \times n}\) has orthonormal rows, and
\(x_{T}\) is segmented into $g$ groups with every element in that group set to one of $\{ -1, 0, 1\}$.
The group-LASSO problem is given by
\begin{equation}
  \label{eq:lasso}
  \min_x \tfrac{1}{2} \|Ax - b\|_2^2 + \lambda \|x\|_{1,2},
\end{equation}
where \(h(x) = \|x\|_{1,2} = \sum_{i = 1}^{g}\|x_{[i]}\|_2\), i.e., the sum of the \(\ell_2\)-norm of the groups.
The groups consisting of all zeros are labeled as ``inactive'', whereas the groups set to $\pm 1$ are ``active''.
We let $m = 512$, $n = 200$ and $\lambda = 10^{-2}$.
We designate $g=5$ such groups of possible 16 (each with 32 elements) to be ``active''.
The noise \(\varepsilon \sim \mathcal{N}(0, 0.01)\).
Thus~\eqref{eq:lasso} has the form~\eqref{eq:nlls}, where \(F(x) = A x - b\).
We set the absolute and relative exit tolerances to be $10^{-4}$ each.
The number of subproblem iterations is capped at 100 for each outer iteration.

\Cref{fig:bpdn} shows the solutions of each algorithm, and \Cref{tab:bpdn} reports the statistics.
All algorithms arrive at approximately the same solution.
{\tt R2} requires the most function evaluations whereas the others require about the same.
\Cref{tab:bpdn} suggests that a tradoff exists between the number of proximal operator evaluations and the number of gradient/Jacobian-vector evaluations.
{\tt TR} takes many proximal iterations, whereas {\tt LMTR} and {\tt LM} take far fewer.
This tradeoff is further exemplified in the next test cases.

\begin{table}[ht]
	\captionsetup{position=top}
  \caption{Group-LASSO~\eqref{eq:lasso} statistics for {\tt R2}, {\tt TR}, {\tt LM}, and {\tt LMTR}, and $h(x) = \|x\|_{1,2}$.
  The $\# \nabla f$ is the number of $J^Tv$ for {\tt LM} and {\tt LMTR}.}
	\label{tab:bpdn}
	\begin{center}
	  \input{code/grouplasso}
	\end{center}
\end{table}

\begin{figure}[ht]
  \centering
  \subfloat[Signal: {\tt R2}]{\label{fig:bpdn_r2_x}\includetikzgraphics[width=.45\linewidth]{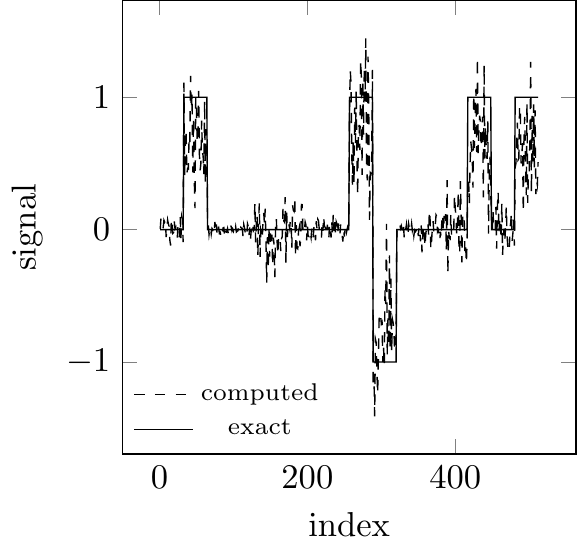}}
  \hfill
  \subfloat[Signal: {\tt TR}]{\label{fig:bpdn_tr_x}\includetikzgraphics[width=.45\linewidth]{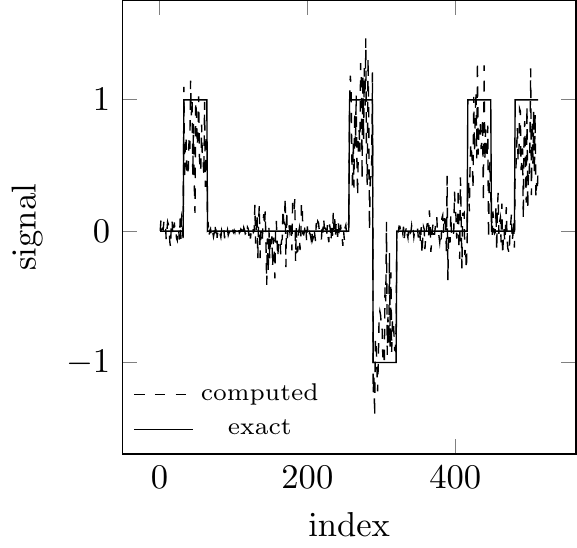}}
  \\
  \subfloat[Signal: {\tt LM}]{\label{fig:bpdn_lm_x}\includetikzgraphics[width=.45\linewidth]{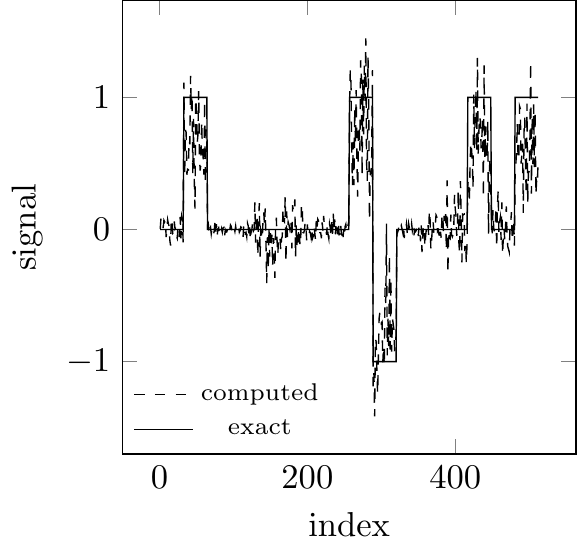}}
  \hfill
  \subfloat[Signal: {\tt LMTR}]{\label{fig:bpdn_lmtr_x}\includetikzgraphics[width=.45\linewidth]{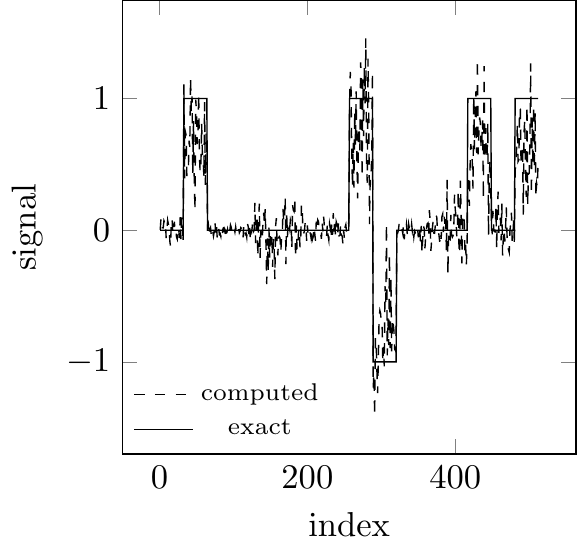}}
  \caption{%
    Group-LASSO~\eqref{eq:lasso} solutions with {\tt R2}, {\tt TR}, {\tt LM}, and {\tt LMTR} with \(h=\lambda\|\cdot\|_{1,2}\).
  }%
  \label{fig:bpdn}
 \end{figure}

We additionally plot descent history in \Cref{fig:objdec-grouplasso}.
The plots are roughly similar, with the trust region methods {\tt TR} and {\tt LMTR} performing the best.

\subsection{Nonlinear support vector machine}

We now solve an image recognition problem of the form~\eqref{eq:nlls}, where
\begin{equation}%
  \label{eq:svm_nonlin}
  F(x) = \mathbf{1} - \tanh(b \odot \langle A, x \rangle),
  \quad
  \mathbf{1} = [1, \ldots, 1]^T,
\end{equation}
\(A\in \R^{m\times n}\), $n = 784$ is the vectorized image size, the number of images is $m = 13007$ in the training set and $m=2163$ in the test set, and \(\odot\) denotes the elementwise product between vectors.
We wish to use this nonlinear SVM to classify digits of the MNIST dataset as either 1 or 7, with all other digits removed.
We additionally impose the condition that the support is sparse, and therefore use $h(x) = \|x\|_{1/2}^{1/2}$ as a regularizer.
Hence, our overall problem is
\begin{equation}
  \label{eq:svmfull}
  \min_x \ \tfrac{1}{2} \| \mathbf{1} - \tanh(b \odot \langle A, x \rangle) \|^2 + \lambda\|x\|_{1/2}^{1/2}
\end{equation}
with $\lambda = 10^{-1}$.
We initialize the problem at $x = \mathbf{1}^n$ so that approximately 50\% of the data is misclassified.
We set the stopping tolerances again to $10^{-4}$ and the maximum number of inner iterations to $100$.

\Cref{fig:svm} shows the solution map of each algorithm, which can be interpreted as the pixels most important in determining whether the image is indeed a 1 or 7.
All algorithms produce a sparse solution; only about ~8\% of pixels in the support vector are nonzero.
The problem is large and nonconvex; hence, the final solutions share pixels but altogether, they are different.
This can be seen in \Cref{tab:svm}, which reports the statistics.
{\tt R2} again requires the most function evaluations.
{\tt TR} requires about 10 times more than {\tt LM} and {\tt LMTR}.
We again observe that a tradoff exists between number of proximal operator evaluations and the number of gradient/Jacobian-vector evaluations.
% {\tt TR} takes many proximal iterations, whereas {\tt LMTR} and {\tt LM} less proximal evaluations but many more $Jv$ iterations.
Here, proximal operator evaluations are cheaper than gradient or $Jv$ evaluations, so wallclock time is higher for {\tt LM} and {\tt LMTR}.

We plot descent history against number of function/residual iterations in \Cref{fig:objdec-svm}.
Here we can see {\tt LM} and {\tt LMTR} performing the best in terms of descent.

\begin{figure}[ht]
 \centering
	\subfloat[Signal: {\tt R2}]{\label{fig:svm_r2_x}\includetikzgraphics[width=.50\linewidth]{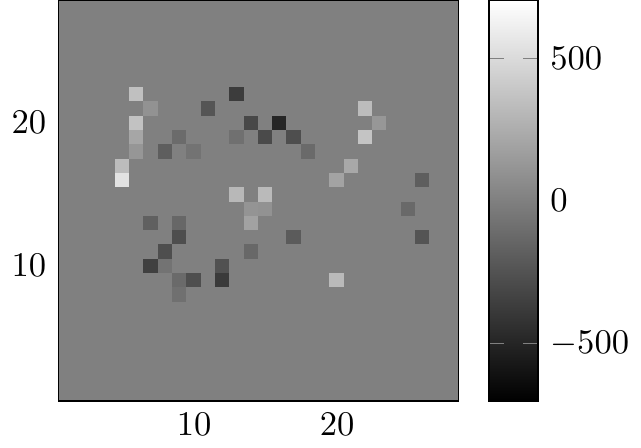}}
  \subfloat[Signal: {\tt TR}]{\label{fig:svm_tr_x}\includetikzgraphics[width=.50\linewidth]{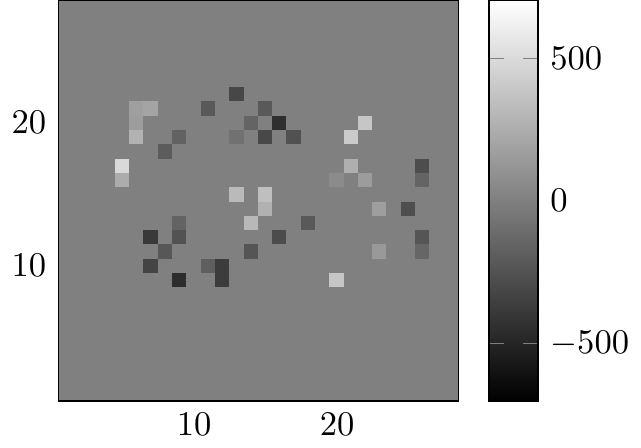}}
 \\
  \subfloat[Signal: {\tt LM}]{\label{fig:svm_lm_x}\includetikzgraphics[width=.50\linewidth]{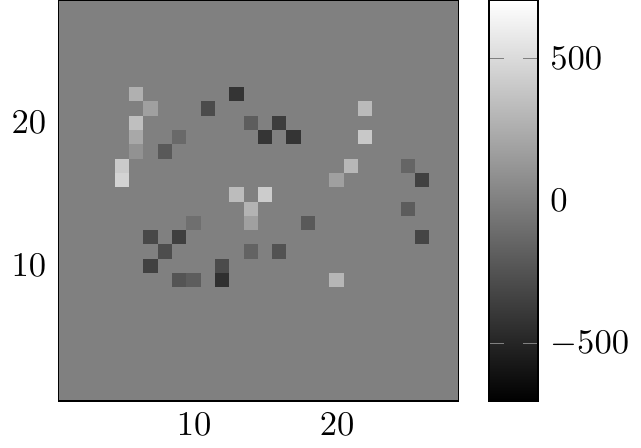}}
	\subfloat[Signal: {\tt LMTR}]{\label{fig:svm_lmtr_x}\includetikzgraphics[width=.50\linewidth]{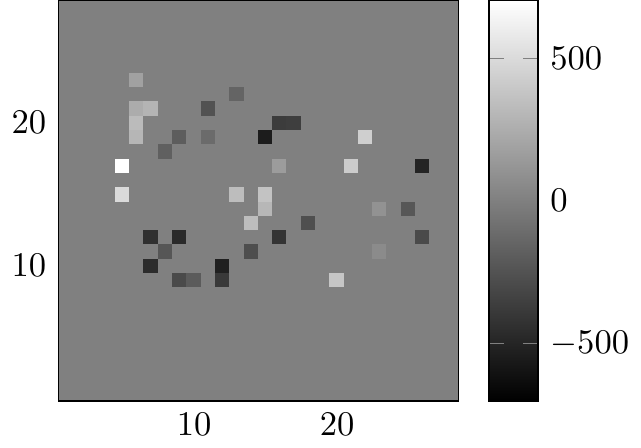}}
	\caption{%
   Nonlinear SVM~\eqref{eq:svmfull} solutions with {\tt R2}, {\tt TR}, {\tt LM}, {\tt LMTR}.
 }%
 \label{fig:svm}
\end{figure}

\begin{table}[ht]
	\captionsetup{position=top}
	\caption{Nonlinear SVM~\eqref{eq:svmfull} statistics for {\tt R2}, {\tt TR}, {\tt LM}, and {\tt LMTR}.
	Training/test error is with respect to the $\ell_2$-norm.}
	\label{tab:svm}
	\begin{center}
	  \input{code/svm}
	\end{center}
\end{table}

\subsection{FitzHugh-Nagumo inverse problem}

The problem has the form~\eqref{eq:nlls}, with \(F: \R^5 \to \R^{2n+2}\) defined as \(F(x) = (v(x) - \bar{v}(\bar{x}), w(x) - \bar{w}(\bar{x}))\), where \(v(x) = (v_1(x), \ldots, v_{n+1}(x))\) and \(w(x) = (w_1(x), \ldots, w_{n+1}(x))\) are sampled values of discretized functions \(V(t; x)\) and \(W(t; x)\) satisfying the \citet{fitzhugh1955mathematical} and \citet{nagumo1962active} model for neuron activation
\begin{equation}%
  \label{eq:fh-ode}
  \frac{\mathrm{d}V}{\mathrm{d}t} = (V - V^3/3 - W + x_1)x_2^{-1},
  \quad
  \frac{\mathrm{d}W}{\mathrm{d}t} = x_2(x_3 V - x_4 W + x_5),
\end{equation}
parametrized by \(x\).
The sampling is defined by a discretization of the time interval \(t \in [0, \, 20]\) and initial conditions \((V(0), W(0)) = (2, 0)\).
The data \((\bar{v}(x), \bar{w}(x))\) is generated by solving~\eqref{eq:fh-ode} with \(\bar{x} = (0, 0.2, 1, 0, 0)\), which corresponds to a simulation of the \citet*{van1926lxxxviii} oscillator.
In our experiments, we use \(n = 100\) and solve
\begin{equation}
  \label{eq:fh_full}
  \min_x \, \tfrac{1}{2} \|F(x)\|_2^2 + \lambda \|x\|_1,
\end{equation}
where \(h(x) = \lambda\|x\|_{1}\) with $\lambda = 10$ to enforce sparsity in the parameters.
Our absolute stopping criteria is $10^{-2}$, whereas our the relative stopping criteria is set to $10^{-4}$.

The solution found by each solver is given in \Cref{tab:fhx}
{\tt TR} has the correct nonzero parameters, but the values are farther off.
The corresponding simulations are shown in \Cref{fig:fh}; each method is able to fit the data.

\begin{table}[ht]
	\captionsetup{position=top}
  \caption{Final parameters for the FH problem~\eqref{eq:fh_full} found by {\tt R2}, {\tt TR}, {\tt LM}, and {\tt LMTR}.}
	\label{tab:fhx}
	\begin{center}
	  \input{code/fhroothalfx}
	\end{center}
\end{table}

\begin{figure}[ht]
 \centering
 \subfloat[Simulation: {\tt R2}]{\label{fig:fh_r2_x}\includetikzgraphics[width=.49\linewidth]{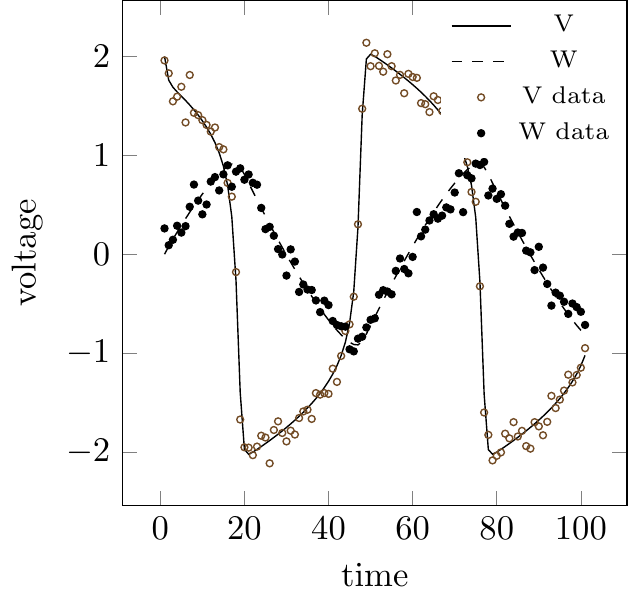}}
 \hfill
 \subfloat[Simulation: {\tt TR}]{\label{fig:fh_tr_x}\includetikzgraphics[width=.49\linewidth]{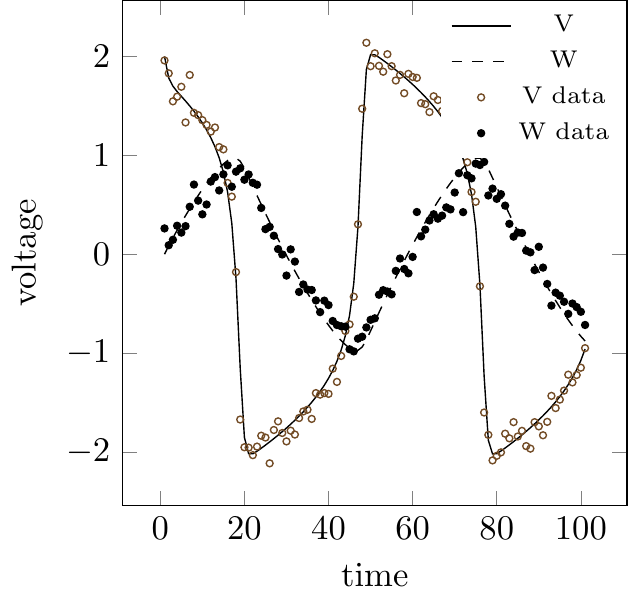}}
 \\
 \subfloat[Simulation: {\tt LM}]{\label{fig:fh_lm_x}\includetikzgraphics[width=.49\linewidth]{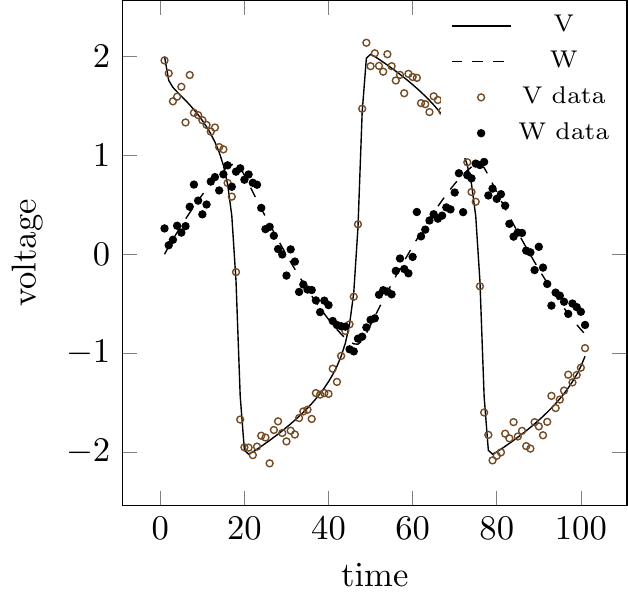}}
 \hfill
 \subfloat[Simulation: {\tt LMTR}]{\label{fig:fh_lmtr_x}\includetikzgraphics[width=.49\linewidth]{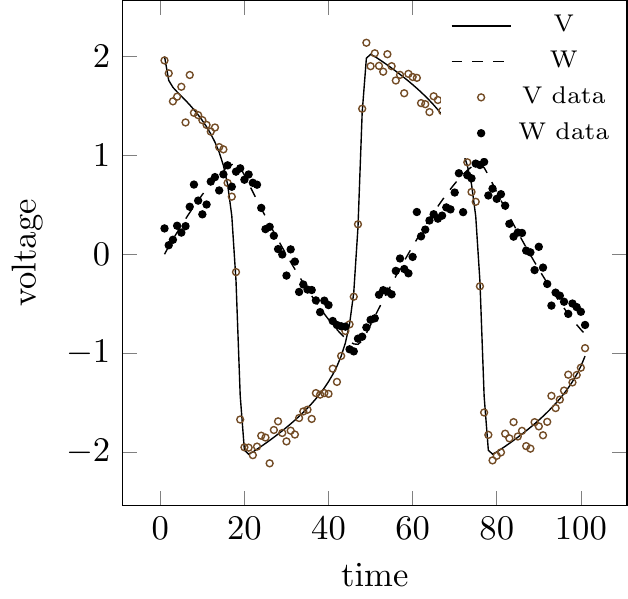}}
 \caption{%
   Simulation of the FH problem~\eqref{eq:fh_full} solutions found by {\tt R2}, {\tt TR}, {\tt LM}, {\tt LMTR}.
 }%
 \label{fig:fh}
\end{figure}

\Cref{tab:fh} reports the statistics for each algorithm, which exhibit the same pattern of results as before.
The final objective values are fairly similar.
{\tt LMTR} uses the smallest amount of objective evaluations, whereas {\tt LM} has a harder time solving \eqref{eq:fh_full}.
Because the gradient of the smooth term in~\eqref{eq:fh_full} is not Lipschitz continuous, we had to set a $\sigma_{\min}$ for both {\tt R2} and {\tt LM}, which increased iteration count.
% \smarttodo{value of $\sigma_{\min}$?}
Similar to the SVM example, we can see that {\tt LM} and {\tt LMTR} take more time than {\tt TR}, which again stems from proximal operators being much cheaper to compute than $Jv$ products for this example.
Notably, {\tt TR} seems to fit the data worse but attain a lower value of the regularizer.

\begin{table}
	\captionsetup{position=top}
  \caption{Statistics for the FH problem~\eqref{eq:fh_full} for {\tt R2}, {\tt TR}, {\tt LM}, and {\tt LMTR}.}
	\label{tab:fh}
	\begin{center}
	  \input{code/fhroothalf}
	\end{center}
\end{table}

Finally, \cref{fig:objdec-fh} shows descent of our objective function value against objective function iteration.
{\tt LMTR} again performs the best, whereas {\tt LM} and {\tt TR} were similar in this metric.
This again enunciates the tradeoff between objective, gradient, and proximal operator expense.
Expensive proximal evaluations would be the limiting factor in {\tt TR} and {\tt R2}; one can think of Total Variation regularization as a test case, since the proximal operator is itself a minimization problem.
% \smarttodo{referees are going to ask for it}

\begin{figure}[ht]
  \centering
  \subfloat[Group-Lasso]{\label{fig:objdec-grouplasso}\includetikzgraphics[width=.5\linewidth]{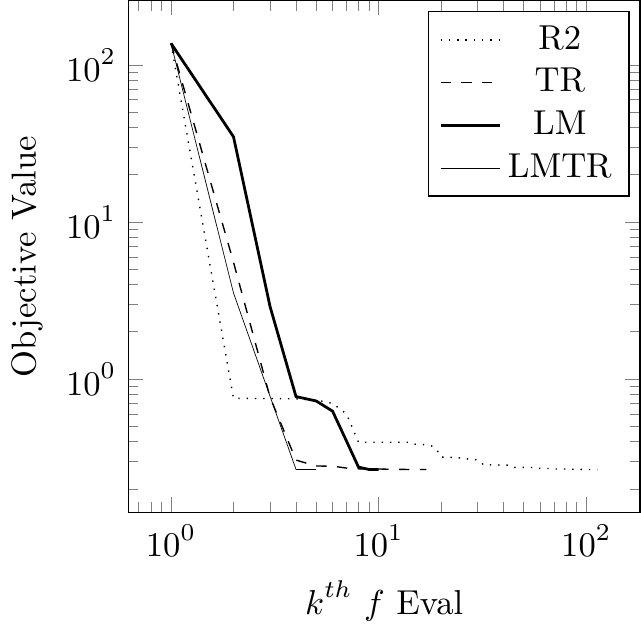}}
  \subfloat[SVM]{\label{fig:objdec-svm}\includetikzgraphics[width=.5\linewidth]{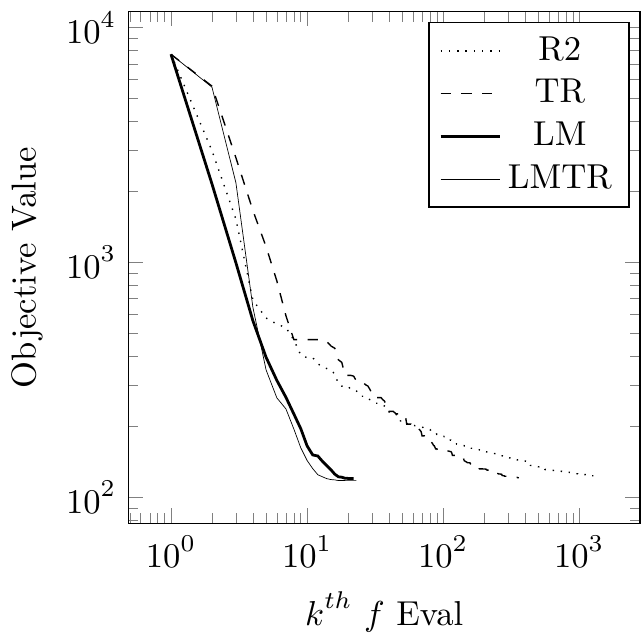}}\\
  \subfloat[FH]{\label{fig:objdec-fh}\includetikzgraphics[width=.5\linewidth]{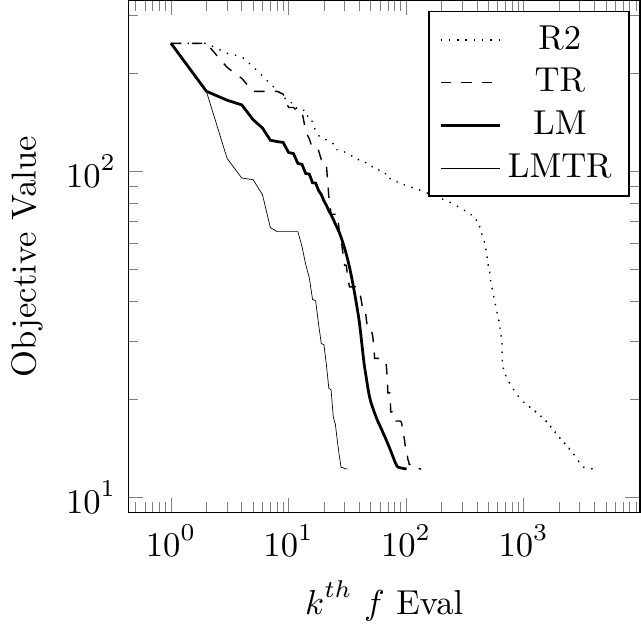}}
  \caption{Objective decrease per objective or residual evaluation. }
\end{figure}

%% file: code/grouplasso.tex
\begin{longtable}{rrrrrrrrr}
  \hline
  Alg & $ f(x) $ & $ h(x) $ & $ (f+h)(x) $ & $ \|x - x_{T}\|_2$ & \# $ f$ & \# $ \nabla f $ & \# $ \prox{}$ & $ t $ (s) \\\hline
  \endfirsthead
  \hline
  Alg & $ f(x) $ & $ h(x) $ & $ (f+h)(x) $ & $ \|x - x_{T}\|_2$ & \# $ f$ & \# $ \nabla f $ & \# $ \prox{}$ & $ t $ (s) \\\hline
  \endhead
  \hline
  \multicolumn{9}{r}{{\bfseries Continued on next page}}\\
  \hline
  \endfoot
  \endlastfoot
  R2 &       0.00 &       0.26 &       0.27 &       0.45 & 113 & 67 & 113 & 0.02 \\
  TR &       0.00 &       0.26 &       0.27 &       0.47 & 17 & 17 & 339 & 2.56 \\
  LM &       0.00 &       0.26 &       0.27 &       0.46 & 10 & 647 & 265 & 0.05 \\
  LMTR &       0.00 &       0.26 &       0.27 &       0.46 & 5 & 327 & 130 & 0.98 \\\hline
\end{longtable}

%% file: code/svm.tex
\begin{longtable}{rrrrrrrrr}
  \hline
  Alg & $ f $ & $ h $ & $ f+h $ & (Train, Test) & \# $f$ & \# $ \nabla f $ & \# $ \prox{}$ & $t $ (s) \\\hline
  \endfirsthead
  \hline
  Alg & $ f $ & $ h $ & $ f+h $ & (Train, Test) & \# $f$ & \# $ \nabla f $ & \# $ \prox{}$ & $t $ (s) \\\hline
  \endhead
  \hline
  \multicolumn{9}{r}{{\bfseries Continued on next page}}\\
  \hline
  \endfoot
  \endlastfoot
  R2 &      57.11 &      66.28 &     123.39 & (99.80, 99.35) & 1359 & 1085 & 1359 & 18.99 \\
  TR &      49.80 &      72.37 &     122.17 & (99.83, 99.26) & 267 & 171 & 10478 & 6.62 \\
  LM &      54.36 &      65.86 &     120.21 & (99.83, 99.35) & 23 & 3567 & 1276 & 24.98 \\
  LMTR &      49.43 &      68.26 &     117.69 & (99.81, 99.12) & 24 & 3925 & 1420 & 44.32 \\\hline
\end{longtable}

%% file: code/fhroothalfx.tex
\begin{longtable}{rrrrr}
  \hline
  True & R2 & TR & LM & LMTR \\\hline
  \endfirsthead
  \hline
  True & R2 & TR & LM & LMTR \\\hline
  \endhead
  \hline
  \multicolumn{5}{r}{{\bfseries Continued on next page}}\\
  \hline
  \endfoot
  \endlastfoot
        0.00 &       0.00 &       0.00 &       0.00 &       0.00 \\
        0.20 &       0.26 &       0.33 &       0.25 &       0.25 \\
        1.00 &       0.84 &       0.70 &       0.86 &       0.85 \\
        0.00 &       0.00 &       0.00 &       0.00 &       0.00 \\
        0.00 &       0.00 &       0.00 &       0.00 &       0.00 \\\hline
\end{longtable}

%% file: code/fhroothalf.tex
\begin{longtable}{rrrrrrrrr}
  \hline
  Alg & $ f $ & $ h $ & $ f+h $ & $ \|x - x_{T}\|_2$ & \# $f $ & \# $ \nabla f $ & \# $ \prox{}$ & $t $ (s) \\\hline
  \endfirsthead
  \hline
  Alg & $ f $ & $ h $ & $ f_h $ & $ \|x - x_{T}\|_2$ & \# $f $ & \# $ \nabla f $ & \# $ \prox{}$ & $t $ (s) \\\hline
  \endhead
  \hline
  \multicolumn{9}{r}{{\bfseries Continued on next page}}\\
  \hline
  \endfoot
  \endlastfoot
  R2 &       1.24 &      10.91 &      12.15 &       1.58 & 4230 & 3428 & 4230 & 40.40 \\
  TR &       1.87 &      10.31 &      12.17 &       1.93 & 134 & 77 & 2452 & 0.67 \\
  LM &       1.20 &      11.03 &      12.23 &       1.55 & 101 & 4236 & 1402 & 20.17 \\
  LMTR &       1.20 &      11.02 &      12.22 &       1.55 & 32 & 2006 & 741 & 10.50 \\\hline
\end{longtable}

%% file: conclusion.tex
\section{Discussion}

Similarly to smooth optimization, exploiting the least-squares structure of \(f\) can decrease significantly the number of outer iterations.
The challenge highlighted by our numerical results, which is the subject of ongoing research, is to either identify a closed-form minimizer of~\eqref{eq:lls-lm-regularized} for relevant choices of \(\psi\), or to devise methods that can produce a higher-quality step than \texttt{R2} with fewer transposed-Jacobian-vector products.
As long as the subproblem solver yields a step satisfying \Cref{asm:step-computation}, our convergence properties and worst-case complexity bounds are guaranteed to hold.
Thus, any improvement in the step computation mechanism will immediately translate into a more efficient solver overall.
In ongoing research, we are exploring other improvements, including inexact evaluations of \(f\) and \(\nabla f\), nonmonotone methods, and inexact evaluation of proximal operators.